\documentclass[12pt, english, a4paper]{article}
\usepackage[latin1]{inputenc}
\usepackage[T1]{fontenc}
\usepackage{babel, graphicx, textcomp, varioref, amsfonts}
\usepackage{amsthm}
\usepackage{amssymb}
\usepackage{amsmath}
\usepackage{mathrsfs}
\usepackage{color}
\usepackage{mathtools}
\usepackage[margin=1in]{geometry}

\usepackage{enumitem}
\usepackage{esint} 
\usepackage[hidelinks]{hyperref}
\usepackage{stmaryrd}
\usepackage{seqsplit}
\usepackage{xstring}
\numberwithin{equation}{section}

\newtheorem{theorem}{Theorem}[section]
\newtheorem{definition}[theorem]{Definition}
\newtheorem{lemma}[theorem]{Lemma}
\newtheorem{proposition}[theorem]{Proposition}
\newtheorem{corollary}[theorem]{Corollary}

\newtheorem{remark}[theorem]{Remark}

\newcommand{\R}{\mathbb{R}} 
\newcommand{\st}{\, \hat{\otimes} \,}

\newcommand{\bH}{V} 
\newcommand{\N}{\mathbb{N}}

\newcommand{\Z}{\mathbf{Z}}

\newcommand{\ZZ}{\mathbb{Z}}

\title{Generalized Burgers equation with rough transport noise}

\author{Antoine Hocquet, Torstein Nilssen and Wilhelm Stannat \thanks{Institute of Mathematics, Technical University of Berlin, Germany, Financial support by the DFG via Research Unit FOR 2402 is gratefully acknowledged.} }

\begin{document}

\maketitle

\abstract{We introduce a new technique for studying well posedness and energy estimates for evolution equations with a rough transport term. The technique is based on finding suitable space-time weight functions for the equations at hand. As an example we study the well posedness of the generalized viscous Burgers equation perturbed by a rough path transport noise.}


\section{Introduction}

Originally the equation by Burgers,
$$
\partial_t u = \nu \partial_x^2 u - u \partial_x u ,
$$
was introduced as a simplified 1-dimensional model of turbulence in the motion of a fluid by neglecting certain terms in, e.g., the Navier Stokes equation. The equation has since been used as a model of many physical phenomena, e.g.\ motion of gas and traffic to name a few, for more examples see \cite{Burgers}.

The motivation for adding noise to this equation is twofold.
One the one hand, the simplified nature of the model could motivate adding randomness to compensate for the neglected terms.  On the other hand, one could hope to model turbulence by adding a highly oscillating noisy term.
There are several ways to introduce noise in the equation, e.g.\ one could imagine a forcing term of stochastic type, i.e.\ additive as studied in \cite{BertiniCancriniJona-Lasinio}, \cite{GyongyNualart}, \cite{DaPratoDebusscheTemam}.

We choose to consider a multiplicative noise in the equation, more specifically on transport form. If we consider the solution of the equation as a velocity field, then the reason for choosing the transport noise can be motivated from the Lagrangian viewpoint. Indeed, assume that the position $\phi_t(x)$ at time $t$ of a fluid particle starting at $x,$ can be decomposed into a regular component and a highly oscillating turbulent term as follows:
$$
\dot{\phi}_t(x) = u_t(\phi_t(x))  - \beta_j(\phi_t(x)) \dot{Z}_t^j,  \quad \phi_0(x) =x.
$$
Here, and for the rest of the paper we use the convention of summation over repeated indicies. Then, the quantity $u$ is  
transported along the trajectories of the above ODE, which motivates the study of the equation
$$
\partial_t u = \nu \partial_x^2 u - u \partial_x u	 + \beta_j \partial_x u \dot{Z}_t^j .
$$

More generally, we consider a multidimensional version of the equation, i.e.\  the generalized Burgers equation
\begin{equation} \label{MainEq}
\partial_t u_t  = \Delta u_t + \mathrm{div} (F(u_t)) + \beta_j \nabla u_t \dot{Z}_t^j
\end{equation}
defined on $\R^d$ and we have let $\nu =1$ for simplicity. Above, we are given an initial condition $u_0 \in L^2(\R^d)$ and sufficiently regular functions $F: \R \rightarrow \R^d$, $\beta_j: \R^d \rightarrow \R^d$ and we assume $Z$ can be lifted to a rough path $\Z = (Z, \ZZ)$.

We choose to work in a variational framework and aim to prove (local) energy estimates of the form
\begin{equation} \label{FiniteEnergySolutions}
\sup_{ 0 \leq t \leq T_{F}} \| u_t \|^2_{L^2} + \int_0^{T_F} \| \nabla u_r \|^2_{L^2} 
\leq C(\| u_0 \|^2_{L^2}),
\end{equation}
for some maximal time $T_F\in(0,T)$ depending on the type of non-linearity considered, and where the right hand side denotes a constant depending on the norm of the initial datum, the vector fields $\beta_j$ and the rough path norm of $\Z$ only.
Solutions satisfying \eqref{FiniteEnergySolutions} for will be refered to as ``finite-energy solutions''. 
Whether such solutions can be global (i.e.\ such that $T=T_F$) depends of course on the choice of $F.$
For simplicity, we will split our analysis into the two following cases of interest:
\begin{itemize}
 \item [-]First we will address the case where $F:\R^d\to\R$ has a bounded derivative, in which case we whall obtain finite-energy solutions on the whole time interval $[0,T].$
 \item [-]
Second, we will consider $d=1$ and $F(u) =  - \frac{1}{2} u^2,$ which then corresponds the classical Burgers non-linearity.
The energy estimates, and therefore the existence and uniqueness of finite-energy solutions, will be then shown to hold locally in time.
\end{itemize}

In the classical setting, i.e.\ when $Z$ is a smooth path, the energy estimates are usually obtained as follows:
multiply formally the equation \eqref{MainEq} by the solution $u$ itself and integrate in space to find
$$
\partial_t \|u_t\|_{L^2}^2 = - 2\|\nabla u_t\|^2_{L^2} - 2(F(u_t),\nabla u_t)  - ( u_t^2, \mathrm{div}(\beta_j)) \dot{Z}_t^j .
$$ 
Integrated in time gives
\begin{align}
\|u_t\|_{L^2}^2  + & 2 \int_0^t \|\nabla u_r\|^2_{L^2} dr  = \|u_0\|_{L^2}^2 - 2 \int_0^t  (F(u_r),\nabla u_r)  - \frac{1}{2}( u_r^2, \mathrm{div}(\beta_j)) \dot{Z}_r^j dr \label{ClassicalEnergyEquality} \\
 & \leq \|u_0\|_{L^2}^2 +   \int_0^t \| \nabla u_r\|^2_{L^2} dr +  \int_0^t \|F(u_r)\|_{L^2}^2  +  \|\mathrm{div}( \beta_j )\|_{\infty} \|u_r\|_{L^2}^2  |\dot{Z}_r^j| dr \notag .
\end{align}
In the first case scenario (i.e.\ when $\|\nabla F\|_{L^\infty}<\infty$), one can then use Gronwall Lemma to obtain the inequality
\begin{align} \label{ClassicalEnergyInequality}
\|u_t\|_{L^2}^2  +  \int_0^t \|\nabla u_r\|^2_{L^2} dr \lesssim  \|u_0\|_{L^2}^2 \exp \Big\{ \|\mathrm{div}(\beta_j)\|_{\infty} \int_0^t |\dot{Z}_r^j|dr \Big\} .
\end{align}
Concerning the classical Burgers equation, a similar bound is obtained by using the fact that the non-linearity is conservative --- it is indeed sufficient to observe that $(u \partial_x u, u) = 0.$
Hence finite-energy solutions turn out to be also global in that case.

Note that the operation of ``testing the equation against itself'' (or what is the same ``multiplying the equation by $u$''),
is not at all straightforward and needs some justification. Namely, if one understands \eqref{MainEq} as an equation in the sense of distributions, one needs to show that, roughly speaking, the space of ``admissible'' test functions contains the solution. 

Such a justification becomes quite more involved as one considers a ``rough signal'' $Z,$ namely such that $Z\in C^{\alpha }$ with $\alpha \leq 1/2.$ In this case, the presence of the unbounded operation $\beta _i\nabla u $ in the rough perturbation increases drastically the number of derivatives needed for a test function to be admissible. Indeed, three derivatives in space will be then needed, a regularity that is far from being satisfied by $u,$ which lies a priori in the energy space $L^\infty([0,T];L^2)\cap L^2([0,T];H^1).$ This is in contrast with the smooth case, where only one derivative is needed.
In addition, the expression \eqref{ClassicalEnergyInequality} does not even make sense for irregular $Z$, so it is also not clear that a Gronwall-like argument could be used.

Both of these problems have been solved for the linear case (\cite{BG}, \cite{DGHT}, \cite{HH}). Using a tensorization method paired with commutator estimates inspired by \cite{DiPernaLions}, it was possible to deduce the equation satisfied by $u_t^2$. In \cite{DGHT} a suitable version of the Gronwall's lemma was introduced.

The main contribution of this paper is to introduce a weighted measure space that is useful for studying non-linear equations. For the case of $\nabla F$ bounded, techniques introduced in \cite{DGHT} would be sufficient; the drift term in the above equation will be sufficiently regular so that the so-called "rough Gronwall lemma" could be applied. For the classical Burgers non-linearity however, one would need a non-linear version of this lemma, i.e.\ a type of rough Bihari-LaSalle inequality. While it is plausible that such an inequality could exist, the method in the present paper allows us to obtain local solutions in time using the classical Bihari-LaSalle inequality.

It should be mentioned that the work \cite{HLN} consider the Navier-Stokes equation in the same framework as the present paper. However, there it is assumed that the vector fields $\beta_j$ are energy preserving, i.e.\ divergence free, which is the physically correct noise for the Navier-Stokes equation. In this case there is no contribution of the noise to the energy. Hence, there is no need to introduce the same measure-change as in the present paper, which is in fact seen from our approach in Proposition \ref{MeasureChange}.

\section{Notation and definitions}

\subsection{H\"{o}lder and Sobolev spaces}

For a fixed time horizon $T> 0,$ we define the simplex $\Delta(T) := \{ (s,t) \in [0,T]^2 :  s \leq t \}$.
Given a Banach space $(E,|\cdot |)$ and $\alpha > 0,$ a mapping $g : \Delta(T) \rightarrow E$ will be said to be $\alpha$-H\"{o}lder continuous provided
for some $L\in(0,T]:$
$$
[g]_{\alpha,L} = \sup_{ (s,t) \in \Delta(T) : |t-s| \leq L  } \frac{|g_{st}|}{|t-s|^{\alpha}} < \infty,
$$
which defines a family of semi-norms $([\cdot ]_{\alpha ,L})_{L\in (0,T]}.$
Although in the sequel we might take $L\neq T$ for convenience, these semi-norms are clearly all equivalent. Therefore, if $L>0$ is clear from the context, we shall abuse notation and write $[g]_\alpha $ instead.
The space of $\alpha$-H\"{o}lder continuous functions will be denoted by $C^{\alpha}_2([0,T];E),$ 
and similarly we let $C^{\alpha}([0,T];E)$ be the space of all $f :[0,T] \rightarrow E$ such that $\delta f \in C^{\alpha}_2([0,T];E),$ where we have defined  
\[
\delta f_{st} : = f_t - f_s.
\]
For a two-parameter function $g: \Delta(T) \rightarrow E$, we also define the second order increment 
$$
\delta g_{s \theta t} := g_{st} - g_{ \theta t} - g_{s \theta}.
$$

We shall work with the usual 
Sobolev spaces $W^{n,p}(\R^d)$ with norm denoted by $| \cdot \nobreak |_{n,p}$. For simplicity we denote by 
$H^n := W^{n,2}(\R^d)$ and the corresponding norm $|\cdot\nobreak  |_{n}$. For smooth and compactly supported functions $f$ and $g$ on $\R^d$, denote by $(f,g) = \int_{\R^d} f(x) g(x) dx$ and by the same bracket the extension of the bilinear mapping
$$
(\cdot, \cdot ) : (W^{n,p}(\R^d))^* \times W^{n,p}(\R^d) \rightarrow \R.
$$

It is easy to see that the scale $(H^n)_{n\in\N}$ posess a family of continuous, linear mappings $J_\eta :L^2(\R^d)\to C^{\infty}(\R^d)$ such that:
\begin{equation}
\label{smoothing_op}
|J^{\eta} \phi|_{n+k} \lesssim \eta^{-k} |\phi|_n 
\quad \text{and}\quad 
|(I - J^{\eta}) \phi|_{n} \lesssim \eta^{k} |\phi|_{n+k} ,
\end{equation}
for $k=0,1,2$ and arbitrary $n\in \N.$
(Consider for instance $J^{\eta} \phi = \rho_{\eta} * \phi$ where $\rho_{\eta}$ is a mollifier such that $\rho_{\eta}(x) = \rho_{\eta}(-x)$.)
In the following, we shall refer to such $(J^{\eta})_{\eta \in [0,1]}$
as a family of \emph{smoothing operators}.

Throughout the paper, we shall restrict our attention to functions
$F: \R \rightarrow \R^d$ that induce a well-defined Nemytskii operator.
Namely, such that $\bar{F} : L^2(\R^d) \rightarrow (H^{-1})^d,$
where
\[
\bar{F}(u)(x) := F(u(x)),\quad x\in \R^d.
\]
For notational convenience, we shall not distinguish between $F$ and $\bar{F}$ in the remainder of the paper. In particular we have that 
\begin{equation}
\label{maps:F}
\mathrm{div} F : L^2(\R^d) \rightarrow H^{-2}
\end{equation}
is a well defined operation, via
\begin{equation}
\label{def:F}
(\mathrm{div} F(u), \phi) = -(F(u), \nabla \phi).
\end{equation}
Two different assumptions on $F:\R\to \R^d,$ both guaranteeing \eqref{maps:F}, will be considered.
First, we assume that $F$ is Lipshitz, i.e.\ there exists some constant (with an abuse of notation) $|\nabla F|_{\infty}$ such that $|F(x)-F(y)|\leq |\nabla F|_{\infty}|x-y|,$ for every $x,y\in\R.$
From the point of view of the Nemytskii operation defined by \eqref{def:F}, this implies in particular the estimate
\begin{equation}
\label{estim:F1}
|\mathrm{div}F(u) - \mathrm{div}F(v) |_{-1} \leq |\nabla F|_{\infty} |u - v |_0 .
\end{equation} 

Second, we shall consider the classical Burgers non-linearity, that is $F(u) = -\frac{1}{2} u^2$ and $d=1.$
In this case, we have for each $\phi \in H^2(\R):$
$$
|(u^2, \partial_x \phi) | \leq |u|_0^2 |\partial_x \phi|_{\infty} \leq |u|_0^2 \sqrt{2} |\partial_x \phi|_0^{1/2} |\partial_x^2 \phi|_0^{1/2} , 
$$
where we have used the well-known interpolation inequality
$$
(\psi(x))^2 = 2 \int_{\infty}^x \psi(y) \partial_x \psi(y) dy   \leq 2|\psi|_0 |\partial_x \psi|_0 \enskip \Longrightarrow  \enskip |\psi|_{0,\infty} \leq \sqrt{2} |\psi|_0^{1/2} |\partial_x \psi|_0^{1/2}. 
$$
Consequently, in this case it holds the following estimate for
the nonlinear operation defined by \eqref{def:F}:
\begin{equation}
\label{estim:F2}
|\partial _x(F(u)-F(v))|_{-2}\leq |u-v|_{0}(|u|_0+|v|_0).
\end{equation}

\subsection{Rough paths, formulation of the equation, and main result} \label{RoughPaths}

Given a smooth $J$-dimensional path $Z = (Z^1, \dots Z^J)$, we can define the iterated integral canonically by
\begin{equation} \label{IteratedIntegral}
\ZZ^{i,j}_{st} := \int_s^t \delta Z^i_{s r} \dot{Z}^j_r dr .
\end{equation}
In the case of an irregular path of, e.g.\ a sample path of the Brownian motion, the above definition does not make sense, since $Z$ is not differentiable but is $\alpha$-H\"{o}lder continuous for $\alpha$ arbitrarily close to $\frac{1}{2}$. In that case however, we could choose for instance to define the integration $\int Z^i dZ^j$ via It\^{o} or Stratonovich integral (yielding then two different definitions for $\ZZ$). 
In both cases, one can show that for almost all sample paths of the Brownian motion we have $\ZZ \in C^{ 2 \alpha}_2([0,T]; \R^{J \times J})$ and it holds the so-called Chen's relations
\begin{equation} \label{ClassicalChensRelation}
\delta \ZZ_{s \theta t}^{i,j} =  
\delta Z_{s \theta}^i \delta Z_{\theta t}^j ,\quad \text{for every}\enskip 0\leq s\leq \theta \leq t\leq T,\enskip
\end{equation}
and each $1\leq i,j\leq J$. 

Motivated by the above, we will say that a pair $\Z = ( Z , \ZZ)$ is a \emph{rough path}
provided \eqref{ClassicalChensRelation} holds,
together with the analytic condition:
\[
(Z,\ZZ)\in C^{\alpha}([0,T] ;\R^J) \times C^{  2 \alpha}_2([0,T]; \R^{J \times J}),
\]
for some $\alpha \in (\frac{1}{3}, \frac{1}{2}].$
We shall say that $\Z$ is \emph{geometric} if there exists a sequence $Z(n)$ of smooth paths such that
$\Z(n) \rightarrow \Z$ with respect to the metric induced by 
$
C^{\alpha}([0,T] ;\R^J)\times C_2^{ 2 \alpha}([0,T]; \R^{J \times J})
$
(in which the first factor is endowed with the norm $\|f\|_{C^\alpha }:=|f_0| + [\delta f]_{\alpha }$)
and where $\ZZ(n)$ is the ``canonical lift'' given by \eqref{IteratedIntegral} with $Z(n)$ instead of $Z.$
We will denote by $\mathscr{C}_g^{\alpha}$ the set of all geometric rough paths.

To formulate the equation, we briefly recall the method introduced in \cite{BG} and further developed in \cite{DGHT}. 
Assume a priori that we have a way of making sense of the integral $\int_s^t  \beta_i  \nabla u_{r} dZ^i_{r}$. We integrate \eqref{MainEq} in time and iterate the corresponding equation into itself to get
\begin{align}
&\delta u_{st}-  \int_s^t \big[\Delta u_{r} + \mathrm{div}F(u_{r})\big] dr \notag
\\
&= \int_s^t  \beta_j   \nabla  u_r dZ_r^j \label{roughIntegral}
\\
& =\int_s^t  \beta_j   \nabla   \left(u_s  + \int_s^r \big[\Delta u_{r_1} + \mathrm{div}F(u_{r_1})\big] dr_1  +   \int_s^r  \beta_i   \nabla u_{r_1} dZ^i_{r_1} \right) dZ_r^j \notag
\\
& = \beta_j \nabla    u_s Z^j_{st}
+  (\beta_j  \nabla)( \beta_i  \nabla)u_s\ZZ^{i,j}_{st} 
 +u^{\natural}_{st} \label{expansion}
\end{align}
where we have defined 
\begin{multline}\label{ExplicitRemainder}
u_{st}^{\natural} 
=   \iiint_{s\leq r_2\leq r_1\leq r\leq t}(\beta_j  \nabla) (\beta_i \nabla )(\beta_k   \nabla )u_{r_2} dZ_{r_2}^k dZ^i_{r_1} dZ^j_r
\\
+ \iint_{s\leq r\leq r_1\leq t}(\beta_j   \nabla )\big[\Delta u_{r_1} + \mathrm{div}F(u_{r_1}) \big]dr_1  dZ^j_r 
\\
+ \iiint_{s\leq r_2\leq r_1\leq r\leq t} (\beta_j  \nabla) (\beta_i  \nabla)\big[\Delta u_{r_2} + \mathrm{div}F(u_{r_2})\big] dr_2 dZ^i_{r_1} dZ^j_r . 
\end{multline}

We now argue that \eqref{ExplicitRemainder} takes 3 derivatives in space but has high time regularity. Indeed, for the first term, assuming only boundedness of $u$ in $L^2(\R^d)$, the term should take values in $H^{-3}$, but as for \eqref{IteratedIntegral}, the extra integral in time should give us a bound of order
$|t-s|^{3 \alpha}$. 
Similary, the second term should take values in $H^{-3}$, but should be bounded by $|t-s|^{1 + \alpha} \lesssim |t-s|^{3 \alpha}$ by assumption on $\alpha$.
The last term takes values in $H^{-4}$, but we should have even higher time regularity, i.e.\ it is bounded by $|t-s|^{1 + 2 \alpha}$. One can then use an interpolation argument to trade the extra derivative in space for the extra time regularity. In fact, in Lemma \ref{Existence} we will show rigorously that $u^{\natural}_{st}$ is bounded by $|t-s|^{ \zeta}$ for some $\zeta > 1$ as a mapping with values in $H^{-3}$.

A posteriori, from the uniqueness in Lemma \ref{sewingLemma}, the expansion \eqref{expansion} gives meaning to the expression \eqref{roughIntegral}.

Define now the operators
$$
A^1_{st}\phi : = \beta_j \nabla \phi \delta Z_{st}^j , \hspace{.5cm} A^2_{st}\phi : = \beta_j \nabla  (\beta_i \nabla \phi ) \ZZ_{st}^{i,j} . 
$$
We have that $A^{i}_{st}$ is bounded from $H^{n+i}$ to $H^{n}$ with an operator norm estimated as
$$
\| A_{st}^i \|_{ \mathcal{L}(H^{n+i} ; H^{n})} \lesssim |t-s|^{i \alpha} 
$$
for $i=1,2$ and $n=3,2,1,0$ and the constant in the above inequality depends on $[\delta Z]_\alpha,[\ZZ]_{2\alpha },$ as well as the norms of the vector fields $\beta_j$ in $C^4_b(\R^d).$ We denote by $[A]_{\alpha}$ the smallest constant satisfying the above bound. Moreover, we have the operator Chen's relation
$$
\delta A_{s \theta t}^2 = A_{\theta t}^1 A_{s \theta }^1\,,\quad \text{for every}\enskip 0\leq s\leq \theta \leq t\leq T.
$$

The above discussion suggests to consider equation \eqref{MainEq} as the following Taylor expansion type equation
\begin{equation} \label{MainEqURD}
\delta u_{st}  = \int_s^t \big[\Delta u_r + \mathrm{div} F(u_r)\big] dr + A_{st}^1 u_s + A_{st}^2 u_s + u_{st}^{\natural}
\end{equation}
on the scale of spaces $(H^n)_n$. More precisely, we give the following notion of a solution for \eqref{MainEq}.

\begin{definition}\label{def:solution}
A bounded path $u :[0,T] \rightarrow H^0$ is said to be a finite-energy solution to \eqref{MainEq} provided that:
\begin{enumerate}[label=(\roman*)]
 \item as a function of $(t,x),$ $u$ belongs to the ``energy space'', namely $u\in L^\infty([0,T];L^2(\R^d))\cap L^2([0,T]; H^1);$
 
 \item the remainder term $u^{\natural}$ which is defined implicitly by \eqref{MainEqURD}, belongs to $C_2^{1+}([0,T];H^{-3}),$ i.e.\ there exists $\zeta > 1$ such that for every $\phi \in H^3,$ and every $(s,t)\in \Delta (T),$ it holds
\end{enumerate}
\begin{equation} \label{RemainderCondition}
|(u_{st}^{\natural},\phi) | \lesssim |t-s|^{\zeta} |\phi|_{3} .
\end{equation}
\end{definition}

Our main result is the following.

\begin{theorem}\label{MainTheorem}
Let $d,J\geq 1$ be integers, fix an arbitrary time horizon $T>0,$ and and initial datum $u_0\in L^2(\R^d).$
Let $\alpha \in(1/3,1/2],$ and consider $\Z\in \mathscr C_g^\alpha([0,T];\R^J) ,$  as well as coefficients
$\beta _j\in C^6_b(\R^d)$ for $j=1,\dots,J.$

Let $F:\R\to\R$ be Lipshitz.
There exists a unique finite-energy solution
to
\begin{equation}
\label{burgers_theorem}
\partial _tu_t=\Delta u_t + \mathrm{div}(F(u_t)) + \beta _j \nabla u_t\dot Z^j
\end{equation} 
on $(0,T]\times \R^d.$

Next, if $d=1$ and $F(u):=-\frac12 u^2,$ there exists a maximal time $T_0\in (0,T]$ such that existence and uniqueness of finite-energy solutions $u$ to
\eqref{burgers_theorem} holds on $(0,T_0)\times \R^d.$
The value of $T_0$ depends only upon the quantities $|u_0|_{0},$ $\|\beta \|_{C^6_b(\R^d)},$ $[\delta Z]_\alpha ,$ $[\ZZ]_{2\alpha }.$
\end{theorem}

Before we proceed to the next section, let us briefly explain the strategy of proof.
\paragraph{Outline of the proof.}
The proof relies on a priori estimates in the energy space $L^\infty([0,T];L^2(\R^d))\cap L^2([0,T];H^1)$ which are either global or local (depending on $F$). 
In Section \ref{section:apriori} we start by showing that if we a priori have a remainder $u^{\natural}$ satisfying \eqref{RemainderCondition} then we can estimate the remainder in terms of $|u|_{L^{\infty}([0,T]; L^2(\R^d))}$ and the unbounded rough driver $(A^1,A^2)$. 
In Section \ref{section:EnergyEstimates} we find a weight function that equilibrates the energy induced by the rough noise on the solution $u$, so that we can replicate \eqref{ClassicalEnergyEquality} and \eqref{ClassicalEnergyInequality}, thus showing energy estimates that only depends on $(A^1, A^2)$ and the initial condition $|u_0|_0$.

To show existence in Section \ref{section:ProofMainTheoremExistence} we approximate the rough path $\Z$ by smooth paths. Once we show that a classical solution gives rise to a solution in the sense of Definition \ref{def:solution}, the a priori energy estimates show that we get a bounded sequence in the energy space. Using a result on strong compactness we can take the limit in the equation.

Uniqueness is proven in Section \ref{section:ProofMainTheoremUniqueness} in a similar way by showing that the energy satisfies a contraction property with respect to initial data.

\subsection{A priori estimates} \label{section:apriori}

In this section we prove a priori estimates for the H\"{o}lder norms related to the solution and the implicit remainder as introduced in Definition \ref{def:solution}. 

Note that in the equation \eqref{burgers_theorem} the ``drift term'' 
\[
\mu_t := \int_0^t \big[\Delta u_r + \mathrm{div} (F(u_r)) \big] dr
\]
is actually Lipschitz continuous from $[0,T]$ to $H^{-2}$. This is indeed a consequence of \eqref{estim:F1}, \eqref{estim:F2}, together with the fact that $u:[0,T]\to L^2(\R^d)$ is bounded by assumption.
In fact, it will be convenient for us to consider the drift term in an abstract way.

Similarly as in Definition \ref{def:solution} we define a weaker notion of a solution to general equations driven by a drift $\mu$ as above.

\begin{definition}\label{def:solutionWeak}
For two paths $u : [0,T] \rightarrow L^2(\R^d)$ and $\mu : [0,T] \rightarrow H^{-3}$ we write
\begin{equation} \label{AbstractEquation}
\partial_t u_t = \partial_t \mu_t + \beta_j \nabla u_t \dot{Z}_t^j,
\end{equation}
provided that $u^{\natural}$ defined by 
$$
(u^{\natural}_{st}, \phi) := ( \delta u_{st}, \phi) - ( \delta \mu_{st}, \phi) - (u_s, [A_{st}^{1,*} + A_{st}^{1,*}] \phi)
$$
satisfies \eqref{RemainderCondition}.
\end{definition}

For this section we focus on a priori estimates and take the drift $\mu$ to be any Lipschitz path with values in $H^{-2}$.

\begin{lemma} \label{aprioriLemma}
Suppose $\mu : [0,T] \rightarrow H^{-2}$ is Lipschitz, and let $u$ satisfy \eqref{AbstractEquation}.

Then there exists $L > 0$ depending $\alpha$ and $\zeta$ only such that  $|t-s| \leq L$ implies
$$
|( \delta u_{st},\phi) | \leq C  \left( |u|_{L^{\infty}([s,t]; H^0) }  + [ \mu]_{\mathrm{Lip}}   + [ A]_{\alpha}  + [ u^{\natural}]_{\zeta} \right) |t-s|^{\alpha^*} |\phi|_1,
$$
and where we let $\alpha^* := \alpha \wedge ( 1- \alpha) \wedge ( \zeta - 2 \alpha) > 0.$
\end{lemma}

\begin{proof}
For any $\phi \in H^1$ we decompose $\phi = J^{\eta} \phi + (I - J^{\eta})\phi$. We get
$$
|(\delta u_{st}, (I-J^{\eta}) \phi) | \lesssim |u|_{L^{\infty}([s,t]; H^0) } | (I-J^{\eta}) \phi |_0  \lesssim |u|_{L^{\infty}([s,t]; H^0) } \eta |\phi|_1 ,
$$
and 
\begin{align*}
|(\delta u_{st}, J^{\eta} \phi) | & \leq [ \mu]_{\mathrm{Lip}} |t-s| |J^{\eta} \phi|_2 + [ A]_{\alpha} |t-s|^{\alpha} |J^{\eta} \phi|_1  
\\
&\quad \quad \quad \quad \quad \quad \quad 
+ [ A]_{\alpha} |t-s|^{2 \alpha} |J^{\eta} \phi|_2 + [ u^{\natural}]_{\zeta} |t-s|^{\zeta} |J^{\eta} \phi|_3
\\
& \lesssim \Big( [ \mu]_{\mathrm{Lip}} |t-s| \eta^{-1} + [ A]_{\alpha} |t-s|^{\alpha}  
\\
&\quad \quad \quad \quad \quad \quad 
+ [ A]_{\alpha} |t-s|^{2 \alpha} \eta^{-1}  + [ u^{\natural}]_{\zeta} |t-s|^{\zeta} \eta^{-2}  \Big)| \phi|_1
\end{align*}
Now, choose $\eta = |t-s|^{\alpha}$ we get
\begin{align*}
|(\delta u_{st}, \phi) | &  \lesssim  \Big( |u|_{L^{\infty}([s,t]; H^0) } |t-s|^{\alpha} + [ \mu]_{\mathrm{Lip}} |t-s|^{1 - \alpha}  
\\
&\quad \quad \quad \quad \quad \quad \quad \quad \quad \quad 
+ [ A]_{\alpha}  + [ u^{\natural}]_{\zeta} |t-s|^{\zeta - 2 \alpha} \Big) |\phi|_1
\\
& \lesssim   \left( |u|_{L^{\infty}([s,t]; H^0) }  + [ \mu]_{\mathrm{Lip}}   + [ A]_{\alpha}  + [ u^{\natural}]_{\zeta} \right) |t-s|^{\alpha^*} |\phi|_1,
\end{align*}
which is the claimed estimate.
\end{proof}

The following lemma is a slightly less general result than the one proved in \cite[Theorem 2.5]{DGHT}. We include a proof since we are working with H\"{o}lder norms instead of $p$-variation.

\begin{theorem} \label{aprioriEstimate}
Consider $u,\mu$ and $\zeta $ as in Lemma \ref{aprioriLemma}.

There exist $L > 0$ and $\zeta^* > 1$ depending only on $\alpha$ and $\zeta,$
such that $|t-s| \leq L$ implies
$$
|(u_{st}^{\natural},\phi) | \leq C  ( [\mu]_{\mathrm{Lip}} [A]_{\alpha}  + ([A]_{\alpha} + [A]^2_{\alpha}) |u|_{L^{\infty}([s,t]; H^0) } )|t-s|^{\zeta^*} |\phi|_3 .
$$
\end{theorem}

\begin{proof}
Applying the second order increment operator to $u^{\natural}$ we get for each $0\leq s\leq \theta \leq t\leq T:$
\begin{align}
\delta u^{\natural}_{s \theta t} & = A_{\theta t}^2 \delta u_{s \theta}  + A^1_{\theta t}( \delta u_{s \theta} - A_{s \theta}^1 u_s)
\label{SmoothInSpace} 
\\
& = (A_{\theta t}^2 + A_{\theta t}^1)\delta \mu_{s \theta}  + (A_{\theta t}^2 A_{s \theta}^1 + A_{\theta t}^2 A_{s \theta}^2 + A_{\theta t}^1 A_{s \theta}^2) u_s   + (A_{\theta t}^1 + A_{\theta t}^2) u_{s \theta}^{\natural} 
\label{SmoothInTime}
\end{align}
where \eqref{SmoothInSpace} comes from Chen's relation and \eqref{SmoothInTime} comes from inserting the equation.

For any $\phi \in H^3$ we decompose $\phi = J^{\eta} \phi + (I - J^{\eta})\phi$. For the smooth part $J^{\eta} \phi ,$ we use \eqref{SmoothInTime} and the properties \eqref{smoothing_op} of the smoothing operator to get 
\begin{align}
 \nonumber
|(\delta u^{\natural}_{s \theta t}, J^{\eta}\phi) | & \leq |t-s| [\mu]_{\mathrm{Lip}} \left|(A_{\theta t}^2 + A_{\theta t}^1)^* J^{\eta} \phi\right|_2 
\\
 \nonumber
 &\quad \quad \quad
+ |u_s|_0 \left|(A_{\theta t}^2 A_{s \theta}^1+ A_{\theta t}^2 A_{s \theta}^2 + A_{\theta t}^1 A_{s \theta}^2)^* J^{\eta} \phi \right|_0
\\
 \nonumber
&\quad \quad \quad \quad \quad \quad \quad 
+ [ u^{\natural}]_{\zeta} |t-s|^{\zeta} \left| (A_{\theta t}^1 + A_{\theta t}^2)^* J^{\eta} \phi \right|_3 
 \\
  \nonumber
 & \lesssim |t-s|^{1 + \alpha}  [\mu]_{\mathrm{Lip}} [A]_{\alpha} | J^{\eta} \phi|_4 +  |u|_{L^{\infty}([s,t]; H^0)} |t-s|^{4 \alpha} [A]_{\alpha}^2 |J^{\eta} \phi |_4
 \\
 \nonumber
 &\quad \quad \quad \quad 
 +  [ u^{\natural}]_{\zeta} |t-s|^{\zeta + 2 \alpha} [A]_{\alpha} | J^{\eta} \phi |_5 
 \\
 \label{SmoothPart}
  & \lesssim \Big( |t-s|^{1 + \alpha}  [\mu]_{\mathrm{Lip}} [A]_{\alpha} \eta^{-1}  +  |u|_{L^{\infty}([s,t]; H^0)} |t-s|^{4 \alpha} [A]_{\alpha}^2  \eta^{-1}
  \\
   \nonumber
  &\quad \quad \quad \quad \quad \quad \quad \quad \quad \quad \quad \quad 
  +  [ u^{\natural}]_{\zeta} |t-s|^{\zeta + 2 \alpha} [A]_{\alpha} \eta^{-2} \Big) |\phi|_3 . 
\end{align}
For the non-smooth part we use \eqref{SmoothInSpace}, yielding
\begin{align}
\nonumber
|(\delta u^{\natural}_{s \theta t}, (I - J^{\eta})\phi) | & \lesssim |u|_{L^{\infty}([s,t]; H^0)} [A]_{\alpha} |t-s|^{2 \alpha} |(I - J^{\eta})\phi|_2
\\
\label{nonSmoothPart}
&\lesssim |u|_{L^{\infty}([s,t]; H^0)} [A]_{\alpha} |t-s|^{2 \alpha} \eta |\phi|_3 . 
\end{align}
Next, let $\eta = \lambda |t-s|^{\kappa}$ with some parameters $\lambda,\kappa>0$ (to be determined later).
From \eqref{SmoothPart} and \eqref{nonSmoothPart}, we get
\begin{align*}
|\delta u^{\natural}_{s \theta t}|_{-3} 
&  \lesssim |t-s|^{1 + \alpha - \kappa} \lambda^{-1}  [\mu]_{\mathrm{Lip}} [A]_{\alpha}
+ |u|_{L^{\infty}([s,t]; H^0)}|t-s|^{4 \alpha - \kappa} \lambda^{-1} [A]_{\alpha}^2 
\\
&\quad \quad \quad \quad 
+  [ u^{\natural}]_{\zeta} |t-s|^{\zeta + 2 \alpha - 2 \kappa}  \lambda^{-2} [A]_{\alpha}
+ |u|_{L^{\infty}([s,t]; H^0)} [A]_{\alpha} |t-s|^{2 \alpha + 2 \kappa} \lambda^2 .
\end{align*}
Choose now $\kappa$ close to $\alpha$ such that
$$
\zeta^* := \zeta \wedge (1 + \alpha - \kappa) \wedge (4 \alpha - \kappa) \wedge (\zeta + 2 \alpha - 2 \kappa) \wedge (2 \alpha + 2 \kappa) > 1.
$$
Then, it holds:
\begin{multline*}
|\delta u^{\natural}_{s \theta t}|_{-3}
\leq C \Big( \lambda^{-1}  [\mu]_{\mathrm{Lip}} [A]_{\alpha} +|u|_{L^{\infty}([s,t]; H^0)}\big(\lambda^{-1} [A]_{\alpha}^2  +  [A]_{\alpha} \lambda^2\big)
\\
+  [ u^{\natural}]_{\zeta^*}  \lambda^{-2} [A]_{\alpha}    \Big) |t-s|^{\zeta^*},
\end{multline*}
where we have used $[ u^{\natural}]_{\zeta} \lesssim [ u^{\natural}]_{\zeta^*}$. 
From Corollary \ref{SewingCorollary} we get
\begin{align*}
[ u^{\natural}]_{\zeta^*} &  \leq C ( \lambda^{-1}  [\mu]_{\mathrm{Lip}} [A]_{\alpha} +  |u|_{L^{\infty}([s,t]; H^0)}( [A]_{\alpha} \lambda^2  + \lambda^{-1} [A]_{\alpha}^2 ) +  [ u^{\natural}]_{\zeta^*}  \lambda^{-2} [A]_{\alpha}   ) .
\end{align*}
Choosing $\lambda$ large enough, namely such that $\lambda^{-2} C [A]_{\alpha} \leq \frac{1}{2},$ and letting $L$ such that $\lambda |t-s|^{\kappa} \in [0,1],$ we obtain the claimed estimate.
\end{proof}

For technical reasons we shall need to estimate the term
\begin{equation} \label{FlatSmoothInSpace}
u^{\flat}_{st} := \delta u_{st} - A^1_{st} u_s .
\end{equation}
The above corresponds to the remainder term in the theory of controlled rough paths, and we see that $u^{\flat} \in C^{\alpha^*}_2([0,T]; H^{-1})$. Moreover, we can write

\begin{equation} \label{FlatSmoothInTime}
u^{\flat}_{st} = \delta \mu_{st} + A^2_{st} u_s + u_{st}^{\natural} ,
\end{equation}
so that also $u^{\flat} \in C^{2 \alpha}_2([0,T]; H^{-3})$. In fact, we can interpolate between these two spaces as follows.

\begin{lemma} \label{aprioriLemmaFlat}
Suppose $\mu : [0,T] \rightarrow H^{-2}$ is Lipschitz, and $u$ solves \eqref{AbstractEquation}.

Then there exists $L > 0$ such that $|t-s| \leq L$ implies
$$
|( u_{st}^{\flat},\phi) | \leq C  \left( |u|_{L^{\infty}([s,t]; H^0) }  + [ \mu]_{\mathrm{Lip}}   + [ A]_{\alpha}  + [ u^{\natural}]_{\zeta} \right) |t-s|^{\zeta - \alpha} |\phi|_2 . 
$$
\end{lemma}

\begin{proof}
For any $\phi \in H^2$ we decompose $\phi = J^{\eta} \phi + (I - J^{\eta})\phi$. 
We get from \eqref{FlatSmoothInSpace}
\begin{align*}
|( u_{st}^{\flat}, (I-J^{\eta}) \phi) | &  \lesssim |u|_{L^{\infty}([s,t]; H^0) } | (I-J^{\eta}) \phi |_0
\\
&\quad \quad \quad 
+ |u|_{L^{\infty}([s,t]; H^0) } [A]_{\alpha}|t-s|^{\alpha} |(I - J^{\eta})\phi|_1  \\
 & \lesssim |u|_{L^{\infty}([s,t]; H^0) }( \eta^2 +  [A]_{\alpha} \eta |t-s|^{\alpha}) |\phi|_2 ,
\end{align*}
and 
\begin{align*}
|( u_{st}^{\flat}, J^{\eta} \phi) | & \leq [ \mu]_{\mathrm{Lip}} |t-s| |J^{\eta} \phi|_2   + [ A]_{\alpha} |t-s|^{2 \alpha} |J^{\eta} \phi|_2 + [ u^{\natural}]_{\zeta} |t-s|^{\zeta} |J^{\eta} \phi|_3 \\
& \lesssim \left( [ \mu]_{\mathrm{Lip}} |t-s|  + [ A]_{\alpha} |t-s|^{2 \alpha}   + [ u^{\natural}]_{\zeta} |t-s|^{\zeta} \eta^{-1}  \right)| \phi|_2 . 
\end{align*}
Now, choose $\eta = |t-s|^{\alpha}$ we get
\begin{align*}
|(u^{\flat}_{st}, \phi) | &  \lesssim  \left((1+ |u|_{L^{\infty}([s,t]; H^0) })(1 + [A]_{\alpha})  + [ \mu]_{\mathrm{Lip}}   + [ u^{\natural}]_{\zeta} \right) |t-s|^{\zeta -   \alpha}  |\phi|_2 \\
\end{align*}
provided $L$ is such that $\eta \in (0,1).$
\end{proof}

We close this section with an It\^{o} formula for tensor products on unbounded rough drivers.

\begin{proposition} \label{ItoTensorProduct}
Assume we have
\begin{equation} \label{uEquation}
\delta u_{st}  = \delta \mu_{st} + A_{st}^1 u_s + A_{st}^2 u_s + u_{st}^{\natural}
\end{equation}
and
\begin{equation} \label{vEquation}
\delta v_{st}  = \delta \nu_{st} + B_{st}^1 v_s + B_{st}^2 v_s + v_{st}^{\natural},
\end{equation}
in the sense of Definition \ref{def:solutionWeak}.

Then the tensor $u \otimes v(x,y) = u(x) v(y)$ satisfies
\begin{equation} \label{uTensorvEquation}
\delta (u\otimes v)_{st}  = \int_s^t \big(\dot{\mu}_r  \otimes v_r + u_r \otimes \dot{\nu}_r\big) dr  + \Gamma_{st}^1 (u \otimes v)_s + \Gamma_{st}^2 (u \otimes v)_s + (u\otimes v)_{st}^{\natural}
\end{equation}
on the scale $H^n(\R^d \times \R^d)$. 
Above we have defined the unbounded rough driver
$$
\Gamma_{st}^1 = A_{st}^1 \otimes I + I \otimes B_{st}^1 \hspace{1cm} \Gamma_{st}^2 = A_{st}^2 \otimes I + I \otimes B_{st}^2 + A_{st}^1 \otimes B_{st}^1,
$$
and by $\dot{\mu}_r$ (respectively $\dot{\nu}_r$) we denote the $r$-almost everywhere defined derivatives of $\mu$ (respectively $\nu$), defined with values in $H^{-2}$.
\end{proposition}

\begin{proof}

Elementary algebraic manipulations give
\begin{align}
\delta (u \otimes v)_{st} & =  u_{s} \otimes \delta v_{st}+ \delta u_{st} \otimes v_{s} +  \delta u_{st}\otimes \delta v_{st} \notag \\
 & =  u_{s} \otimes \delta v_{st}+ \delta u_{st} \otimes v_{s} +  ( u_{st}^{\flat} + A_{st}^1u_s) \otimes ( v^{\flat}_{st} + B_{st}^1) \notag \\
 & =  \int_s^t \big[\dot{\mu}_r  \otimes v_r + u_r \otimes \dot{\nu}_r\big] dr  + \Gamma_{st}^1 (u \otimes v)_s + \Gamma_{st}^2 (u \otimes v)_s + (u\otimes v)_{st}^{\natural} \label{eq:tensor difference},
\end{align}

where we have defined  
\begin{align} \label{eq:tensorEqn2}
(u \otimes v)^{\natural}_{st} & :=  -  \int_s^t \big[ \dot{\mu}_{r} \otimes \delta v_{sr} +\delta u_{sr} \otimes \dot{\nu}_r\big] dr +  u_{st}^{\natural} \otimes v_s + u_s \otimes v^{\natural}_{st}   +  u_{st}^{\flat} \otimes v_{st}^{\flat} \notag \\
 & +  u_{st}^{\flat} \otimes B_{st}^1 v_s +  A_{st}^1 u_s \otimes v_{st}^{\flat} . 
 \end{align}
The proof is done once we can show that  
$$
|((u \otimes v)^{\natural}_{st},\Phi)| \le  C |\Phi|_{H^3(\R^d \times \R^d)} |t-s|^{\zeta^*} 
$$
for some some $\zeta^* > 1$. 

To do so we fix $\Phi \in H^3(\R^d \times \R^d)$, and examine the expression \eqref{eq:tensorEqn2} directly. Since many of the arguments are symmetric, we will only consider some of them.
Using that $H^{n+m}(\R^d \times \R^d) \simeq H^{n} \otimes H^{m}$ valid for $m,n \geq 0,$
we get 
\begin{align*}
&\text{\textbullet}\quad 
\left| \int_s^t (\delta u_{st} \otimes \dot{\nu}_r,  \Phi) dr  \right| \leq \int_s^t |\delta u_{sr}|_{-1}  [\nu]_{\mathrm{Lip}}  dr |  \Phi |_{H^1 \otimes H^2} \lesssim |t-s|^{1 + \alpha^*} |\Phi|_{H^3(\R^d \times \R^d)}, 
\\
&\text{\textbullet}\quad 
|  (u^{\natural}_{st} \otimes v_s ,  \Phi)|\le |v_s|_{0} |t-s|^{\zeta} | \Phi|_{H^3 \otimes H^0}
\lesssim  |t-s|^{\zeta} |\Phi|_{H^3(\R^d \times \R^d)} 
\\
&\text{\textbullet}\quad 
|(u_{st}^{\flat} \st B_{st}^1 v_s  , \Phi) |
\lesssim |v_s|_0 |t-s|^{\zeta }  | \Phi|_{H^2 \otimes H^1} \lesssim  |t-s|^{\zeta }  | \Phi|_{H^3( \R^d \times \R^d)} ,
\intertext{while:}
&\text{\textbullet}\quad 
|(u_{st}^{\flat} \otimes v_{st}^{\flat} , \Phi) | 
\lesssim |t-s|^{\zeta - \alpha} |t-s|^{\alpha^*} | \Phi |_{H^2 \otimes H^1} = |t-s|^{\zeta - \alpha + \alpha*}| \Phi |_{H^3(\R^d \times \R^d) }.
\end{align*}
Note that in the last estimate, we have used that $u^{\flat} \in C_2^{\zeta - \alpha}([0,T]; H^{-2})$ and $v^{\flat} \in C_2^{\alpha^*}([0,T]; H^{-1})$.
Letting $\zeta^* := (1 + \alpha^*)\wedge \zeta \wedge (\zeta - \alpha + \alpha*),$ the result follows. 
\end{proof}

From the above proof we can extract the following result.

\begin{lemma}[Product Formula] \label{ItoInnerProduct}
Suppose $v \in L^{\infty}([0,T]; H^3)$, $v^{\flat} \in C_2^{2 \alpha}([0,T]; H^2)$, $v^{\natural} \in C_2^{3 \alpha}([0,T]; H^0)$ and $\dot{\nu} \in L^{\infty}([0,T]; H^2)$. Then we have the scalar equality
\begin{align*}
\delta (u,v)_{st}  & = \int_s^t (\dot{\mu}_r, v_r) + (u_r, \dot{\nu_r}) dr + (u_s, \bar{\Gamma}_{st}^{1} v_s + \bar{\Gamma}_{st}^{2} v_s ) + (u,v)_{st}^{\natural}
\end{align*}
where we have defined the operators
$$
\bar{\Gamma}_{st}^{1} = A^{1,*}_{st} + B_{st}^1,  \quad  \bar{\Gamma}_{st}^{2} = A^{2,*}_{st} + B_{st}^2 +  A_{st}^{1,*} B_{st}^1,
$$
and we have
$$
|(u,v)_{st}^{\natural} | \lesssim |t-s|^{\zeta^*},
$$
for some $\zeta^* >1$.
\end{lemma}

\begin{proof}
The inner product $(\cdot, \cdot)$ on $L^2(\R^d)$ can be extended to a linear mapping 
$$
(\cdot, \cdot) : H^{-n} \otimes H^n \rightarrow \R 
$$
for every $n$. 
By the assumptions on $v$, $v^{\flat}$ and $\dot{\nu}$ we can apply this linear transformation to both expressions for $(u \otimes v)^{\natural}$ in the proof of Proposition \ref{ItoTensorProduct}, giving the result. 
\end{proof}

\section{Energy estimates} \label{section:EnergyEstimates}

In this section we establish the a priori estimates for the finite-energy solutions to \eqref{MainEq}. Let us first make a heuristic derivation of our approach. In the classical setting, i.e.\ when $Z$ is a smooth path we can find the equation for $u_t^2$ by the chain rule;
$$
\partial_t u^2_t = \Delta u^2_t - 2|\nabla u_t|^2   + 2 u_t \mathrm{div}(F(u_t))+ \beta_j \nabla u^2_t \dot{Z}_t^j .
$$

For a sufficiently regular function $m : [0,T] \times \R^d \rightarrow \R$ we get from the product formula
\begin{align*}
\partial_t (u^2_t, m_t) 
&  = (u_t^2, \Delta m_t) - 2( |\nabla u_t|^2, m_t)   - 2( F(u_t), \nabla(u_t m_t))  
\\
& \quad \quad 
- (u_t^2, \mathrm{div}( \beta_j m_t) ) \dot{Z}_t^j + (u_t^2, \partial_t m_t)
\\
&  = (u_t^2, \Delta m_t+ \partial_t m_t - \mathrm{div}( \beta_j m_t)  \dot{Z}_t^j  ) - 2( | \nabla u_t|^2, m_t)   - 2( F(u_t), \nabla(u_t m_t)).
\end{align*}

Thus, if we can solve the backward equation 
\begin{equation} \label{mFunction}
\partial_t m_t = - \Delta m_t+ \mathrm{div}( \beta_j m_t) \dot{Z}_t^j,\quad m_T = 1,
\end{equation}
integrating in time, we get the following weighted energy equality
\begin{equation} \label{EnergyEquality}
 (u_t^2, m_t) + \int_0^t 2( |\nabla u_r|^2, m_r) dr = |u_0|^2_{0} - 2 \int_0^t ( F(u_r) , \nabla( u_r m_r)) dr .
\end{equation}
We will show that \eqref{mFunction} has a solution and how to make sense of the dual pairing $(u^2_t,m_t)$. In addition we show that $m$ is uniformly bounded away from 0 and $\infty$, which will lead to the energy estimates. In fact we have the following theorem.

\begin{theorem} \label{EnergyEqualityTheorem}
Suppose $\Z$ is a geometric rough path and $\beta_j \in C_b^6(\R^d)$ for all $j$. Given $F$ with bounded derivative we have the following estimate, for every  $u,$ finite-energy solution to \eqref{MainEq}:
\begin{equation} \label{FBoundedEnergy}
\sup_{t\in[0,T]}|u_t|_0^2 + \int_0^T  |\nabla u_r|_0^2 dr  \leq C|u_0|_0^2 ,
\end{equation}
where the above constant depends on the quantities $T,|\nabla F |_{\infty},  [\delta Z ]_{\alpha},  [\ZZ ]_{2 \alpha} $ and  $\|\beta \|_{C_b^6},$ but not on $u.$

Assuming now that $d=1$ and $F(u)=-\frac{1}{2}u^2$
there exists a time $T_0$ and a constant $C,$ 
both depending on $\|\beta \|_{C^6_b}, [\delta Z ]_{\alpha},  [\ZZ ]_{2 \alpha}$ and $|u_0|_{L^2}$ only,
such that every finite-energy solution $u$ to
$$
\partial_t u = \partial _{x}^2u - u\partial _xu + \beta_j \partial_x u \dot{Z}_t^j
$$
satisfies
\begin{equation} \label{BurgersBoundedEnergy}
\sup_{0 \leq t \leq T_0} |u_t |_0^2 + \int_0^{T_0} |\nabla u_r |_0^2 dr
\leq C .
\end{equation}
\end{theorem}

The remaining pararagraphs of this section will be devoted to the proof of Theorem \ref{EnergyEqualityTheorem}.
It consists mainly in justifying the relation \eqref{EnergyEquality}.

\subsection{Tensored equation}

Although testing the solution against itself is a priori not a well defined operation, the tensor can always be defined canonically as demonstrated in Proposition \ref{ItoTensorProduct}. As an immediate corollary we get the following. We define the symmetric tensor $a \st b = \frac{1}{2}( a \otimes b + b \otimes a)$.

\begin{corollary}
The tensor $u^{\otimes 2}_t(x,y) := u_t(x) u_t(y)$ satisfies
$$
\partial_t u_t^{\otimes 2} = 2 u_t \st \Delta u_t +2 u_t \hat{\otimes} \mathrm{div}(F(u_t)) + 2( I \st  \beta_j \nabla ) u_t^{\otimes 2} \dot{Z}_t^j
$$
on the scale $H^{n}(\R^d \times \R^d)$, i.e.\ if we define 
$$
\Gamma^1_{st} := 2 A_{st}^1 \st I \hspace{.5cm} \Gamma^2_{st} := 2 A_{st}^2 \st I + A_{st}^1 \otimes A_{st}^1
$$
we have that
\begin{equation} \label{tensorEq}
\delta u_{st}^{\otimes 2} = \int_s^t 2 u_r \hat{\otimes} \Delta u_r  + 2 u_r \hat{\otimes} \mathrm{div}(F(u_r)) dr + (\Gamma_{st}^1 + \Gamma_{st}^2) u_s^{\otimes 2} + u_{st}^{\otimes 2, \natural} 
\end{equation}
where $|(u_{st}^{\otimes 2, \natural} , \Phi) | \leq C|t-s|^{\zeta} |\Phi |_{H^{3}(\R^d \times \R^d)}$ for some $\zeta > 1$.
\end{corollary}

\subsection{Lyapunov weight function}

In this section we introduce an auxiliary function that will allow us to approximate the solution of \eqref{mFunction} in the dual pairing \eqref{EnergyEquality}. 
This method can be thought of as a way of approximating the Dirac-delta on the diagonal $\{x=y\}$ tailored to the equation at hand.

The main tool that we will use is the Feynman-Kac formula extended to the rough path setting as demonstrated in \cite{DFS} and later explored in the $L^2(\R^d)$ setting in \cite{FNS}. Let us briefly explain the idea.

Let $\phi_{t,s}(x)$ denote the flow of the SDE
\begin{equation} \label{SmoothSDE}
d X_s = \sigma_i(X_s) dB^i_s + \beta_j(X_s) \dot{Z}_s^j ds ,
\end{equation} 
where $B$ is an $I$-dimensional Brownian Motion, defined on some probability space and $\sigma_i :\R^d\to\R^d$ is bounded and measurable for each $i\in\{1,\dots,I\}$.
If $Z$ is a smooth path, it is well known that the above equation gives rise to the solution to the backward equation
\begin{equation} \label{SmoothKolmogorov}
\partial_t v  = -\frac{1}{2} \sigma_{i}^k \sigma_{j}^k  \partial_{i }\partial_{j} v + \mathrm{div}( \beta_j v) \dot{Z}_t^j \enskip ,
\end{equation}
with a given final condition $v_T$.
In fact, the classical Feynman-Kac formula tells us that we can represent the solution as
\begin{equation} \label{SmoothFeynmanKac}
v_t(x) = E\Big[ v_T(\phi_{t,T}(x))  \exp \Big\{ \int_t^T \mathrm{div} \beta_j(\phi_{t,s}(x))\dot{Z}_s^j ds \Big\} \Big] .
\end{equation}

In \cite{DFS,FNS} it is shown that the expressions, \eqref{SmoothSDE}, \eqref{SmoothKolmogorov} and  \eqref{SmoothFeynmanKac} extend to case where $Z$ is a geometric rough path and the relationship between these equations is still preserved in the limit. 
The solution $X$ 
to \eqref{SmoothSDE} has however no meaning a priori, even in the rough path sense.
The twist is that the latter should be interpreted in a pathwise way as the solution of the rough differential equation
\begin{equation} \label{RoughSDE}
dX_s = V_j(X_s) d \tilde{\Z}_s^j
\end{equation}
where $\tilde{\Z}$ denotes now an enhancement \emph{of the pair} $(B,Z)$, that is
$$
\tilde{Z}_s = 
\left( \begin{array}{c}
B_s \\
Z_s \\
\end{array} \right)
\hspace{.3cm} \textrm{ and } \hspace{.3cm}
\tilde{\ZZ}_{st} = 
\left( \begin{array}{cc}
\mathbb{B}_{st}  & \int_s^t \delta Z_{sr} dB_r\\
\int_s^t \delta B_{sr} dZ_r & \ZZ_{st} \\
\end{array} \right)
$$
and $V_j(x) = \sigma_j(x)$ for $j =1, \dots , d$ and $V_j(x) = \beta_j(x)$, for $j=d+1, \dots d + J$. The integral $\int_0^{t } V_j(X_r) d \tilde{\Z}_r$ is to be understood as $I_{t}$ from Lemma \ref{sewingLemma} with the local expansion
\begin{equation} \label{RoughIntegralExpansion}
G_{st} = V_j(X_s) \tilde{Z}_{st}^j +  (\partial_i V_j)(X_s)  V_i(X_s) \tilde{ \ZZ}^{i,j}_{st}\,.
\end{equation}

\begin{remark}
The reader should note that we did not define the iterated integrals of the Brownian motion in the definition of $\tilde{\ZZ}$. As it turns out in the present context, it does not matter how we choose $\mathbb{B}$ 
(It\^{o} or Stratonovich integral will do). This is because we restrict our attention to $\sigma=\mathrm{cst} \equiv1,$ in which case it is easily seen, considering \eqref{RoughIntegralExpansion}, that $X$ is independent of the choice of enhancement we make for $B.$
\end{remark}

Equation \eqref{SmoothFeynmanKac} now becomes
\begin{equation} \label{RoughFeynmanKac}
v_t(x) = E\Big[ v_T(\phi_{t,T}(x))  \exp \Big\{ \int_t^T \mathrm{div} \beta_j(\phi_{t,s}(x)) d\tilde{\Z}_s^j  \Big\} \Big] ,
\end{equation}
where $\phi_{t,s}(x)$ denotes the flow of \eqref{RoughSDE}. 
The above formula reveals the spatial regularity we can expect from the solutions to \eqref{SmoothKolmogorov}. Namely the smoothness of $\sigma_j$, $\beta_j$ should be inherited by the flow, $\phi_{t,s}(\cdot)$, and in turn $v$ via \eqref{RoughFeynmanKac}. 

Based on this we will introduce two equations and their corresponding solutions that will be helpful for obtaining energy estimates. For a proof of the following result we refer to \cite{FNS}.

\begin{proposition} \label{TensorMeasureChangeProposition}
Given a function $M_T \in H^6(\R^d \times \R^d)$ and vector field $\beta_j \in C^6_b(\R^d)$ there exists a solution $M$ to 
\begin{equation} \label{TensorMeasureChangeEquation}
\partial_t M_t = - (\nabla_x + \nabla_y)^2 M_t  +  [ I \otimes  \mathrm{div}_y(\beta_j \cdot)  + \mathrm{div}_x( \beta_j \cdot) \otimes I] M_t \dot{Z}_t^j,
\end{equation}
in $[0,T]\times\R^d \times \R^d,$
with final condition $M_T$, i.e.\
\begin{equation} \label{TensorMeasureChangeURD}
\delta M_{st} = - \int_s^t (\nabla_x + \nabla_y)^2 M_r dr   - \Gamma_{st}^{1,*} M_s +  \Gamma_{st}^{2,*} M_s  + M_{st}^{\natural}
\end{equation}
where $|(M_{st}^{\natural}, \Phi)| \lesssim |t-s|^{\zeta} |\Phi|_{H^3(\R^d \times \R^d)},$
and above we have used the short-hand notation $(\nabla_x + \nabla_y)^2 = (\mathrm{div}_x + \mathrm{div}_y)(\nabla_x + \nabla_y).$

The solution, which is continuous as a path with values in $H^4(\R^d \times \R^d),$ is given explicitly by
\begin{equation} \label{TensorMeasureChangeSolution}
 M_t(x,y) = E\Big[ M_T\big(\phi_{t,T}(x), \phi_{t,T}(y)\big)  \exp \Big\{ \int_t^T \mathrm{div} \beta_j(\phi_{t,s}(x)) + \mathrm{div} \beta_j(\phi_{t,s}(y)) d\tilde{\Z}_s^j \Big\} \Big],
\end{equation}
and in fact it holds
$|M_{st}^{\natural}|_{H^3(\R^d \times \R^d)}  \lesssim |t-s|^{\zeta} $
and $|M_{st}^{\flat}|_{H^2(\R^d \times \R^d)} \lesssim |t-s|^{2 \alpha}$.
\end{proposition}

\begin{proof}

In \cite{FNS} it is shown that $M$ is a bounded solution in $H^3(\R^d \times \R^d)$ with the appropriate bounds on $M^{\natural}$ and $M^{\flat}$ under the weaker assumption that $\beta_j \in C_b^5(\R^d)$. When $\beta_j \in C_b^6(\R^d)$ it is easy to see that the solution is actually bounded in $H^4(\R^d \times \R^d)$, and the result follows.
\end{proof}

\begin{remark}
The expert reader would notice that the second order operator in \eqref{TensorMeasureChangeEquation} is degenerate in a way that one only expects a regularization effect in the ``direction parallel to the diagonal''. Therefore, the regularity of the solution comes in general only from the formula \eqref{TensorMeasureChangeSolution} and the assumption $\beta_j \in C^6_b(\R^d).$

Moreover, although for our purpose it is not needed, it turns out that the solution to \eqref{TensorMeasureChangeEquation} is unique in the class of solutions described in Proposition \ref{TensorMeasureChangeProposition}.
\end{remark}

Since the solution $M$ is smooth in space, we may take the inner product between \eqref{tensorEq} and \eqref{TensorMeasureChangeEquation}. 

\begin{lemma}
For any given $M_T \in H^6(\R^d \times \R^d)$ we have the following

\begin{equation} \label{TensorInnerProduct}
\partial_t (u^{\otimes 2}_t, M_t) = -2 (\nabla u_t \otimes \nabla u_t, M_t) + 2( u_t \st \mathrm{div} F(u_t), M_t) . 
\end{equation}
\end{lemma}

\begin{proof}
From Proposition \ref{ItoInnerProduct} we get
\begin{equation} \label{TensorInnerProductWithRemainder}
\delta (u^{\otimes 2} , M)_{st}   =- \int_s^t 2 (\nabla u_r \otimes \nabla u_r, M_r) + 2( u_r \st \mathrm{div} F(u_r), M_r) dr +  (u^{\otimes 2}, M)^{\natural}_{st} 
\end{equation}
where 
$$
|(u^{\otimes 2}, M)^{\natural}_{st}  | \lesssim |t-s|^{\zeta^*}
$$
for some $\zeta^* > 1$. Since the other terms in \eqref{TensorInnerProductWithRemainder} are increments in time, so is $(u^{\otimes 2}, M)^{\natural}_{st} $ which implies that it is zero. The result follows.
\end{proof}

\begin{remark}
It is possible to justify \eqref{TensorInnerProduct} under the assumption that $\beta_j \in W^{3, \infty}(\R^d)$. This would however require that we introduce a much bigger machinery, and it is not the aim of this paper.
\end{remark}

\bigskip

Similarly we have the following.

\begin{proposition} \label{MeasureChange}
Assume $\beta_i \in C^6_b(\R^d)$. Then the backward equation \eqref{mFunction} 
has a unique solution in $C^4_b(\R^d)$ given by
\begin{equation} \label{mFunctionFeynmanKac}
m_t(x) = E\Big[  \exp \Big\{ \int_t^T \mathrm{div} \beta_j(\phi_{t,s}(x)) d\tilde{\Z}_s^j  \Big\} \Big] .
\end{equation}
Moreover, 
$$
0  < \inf_{r,x} m_r(x) \leq \sup_{r,x} m_r(x) < \infty
$$ 
uniformly on bounded sets in $\mathscr{C}_g^{\alpha}$.
\end{proposition}

\begin{proof}
The upper bound on the representation \eqref{mFunctionFeynmanKac} was first proved in \cite{DOR}. Its relationship to the equation \eqref{mFunction} was shown in \cite{DFS} and it is in fact the unique solution in $C_b^4(\R^d)$. 

To see also the lower bound on the solution, consider the following computations
\begin{align*}
1 &  = E[1]^2  =  E\Big[ \exp \Big\{ \frac{1}{2} \int_r^T \mathrm{div} \beta_j(\phi_{r,s}(x)) d\tilde{\Z}_s^j \Big\} \exp \Big\{  - \frac{1}{2} \int_r^T \mathrm{div} \beta_j(\phi_{r,s}(x)) d\tilde{\Z}_s^j \Big\}\Big]^2 
\\
 & \leq  E\Big[ \exp \Big\{  \int_r^T \mathrm{div} \beta_j(\phi_{r,s}(x)) d\tilde{\Z}_s^j \Big\} \Big] E\Big[ \exp \Big\{ - \int_r^T \mathrm{div} \beta_j(\phi_{r,s}(x)) d\tilde{\Z}_s^j \Big\} \Big],
\end{align*}
so the lower bound on $m$ will follow if we can show the upper bound on 
$$
\tilde{m}_r(x) := E\Big[ \exp \Big\{ - \int_r^T \mathrm{div} \beta_j(\phi_{r,s}(x)) d\tilde{\Z}_s^j \Big\} \Big] .
$$
The latter is in turn the solution to 
$$
\partial_t\tilde{m}_t = - \Delta \tilde{m}_t +  \big(\beta_j \nabla \tilde{m}_t - \mathrm{div} \beta_j  \tilde{m}_t \big)\dot{Z}_t^j, \quad  \tilde{m}_T = 1 .
$$
The result follows.
\end{proof}

We now choose $M^N$ to be the solution to \eqref{TensorMeasureChangeEquation} with final condition $M_T^N(x,y) := \sum_{n \leq N} e_n(x) e_n(y)$, where $ \{e_n\}_{n \geq 1}$ is an orthonormal basis of $L^2(\R^d)$ such that $e_n \in H^6$ (e.g. the Hermite functions). The solution is given by
$$
M_t^N(x,y)  =  \sum_{n \leq N} E\Big[ e_n(\phi_{t,T}(x)) e_n(\phi_{t,T}(y))   \exp \Big\{ \int_t^T \mathrm{div} \beta_j(\phi_{t,s}(x)) + \mathrm{div} \beta_j(\phi_{t,s}(y)) d\tilde{\Z}_s^j \Big\} \Big] .
$$

Clearly, for the final time $T$ we have for any $f \in L^2(\R^d)$
\begin{align*}
\int_{\R^d \times \R^d} f(x)f(y) M^N_T(x,y) dx dy & = \sum_{ n \leq N } \left| \int_{\R^d} f(x) e_n(x) dx \right|^2 \rightarrow \sum_{n \geq 1} |(f,e_n)|^2 = |f|_0^2,
\end{align*}
as $N \rightarrow \infty$.

We now show that a similar result holds in the space weighted by $m_t$.

\begin{lemma}
For any $f,g \in L^2(\R^d)$ we have 
$$
\lim_{N \rightarrow \infty} \iint_{\R^d \times \R^d} f(x)g(y) M^N_t(x,y) dx dy = \int_{\R^d} f(x)g(x) m_t(x) dx
$$
\end{lemma}

\begin{proof}
We prove the lemma for $f = g$, the general result follows from the parallelogram law.
Begin by noticing that for all $f \in L^2(\R^d)$ we have $P$-a.s. that $f \circ \phi_{t,s}^{-1} \in L^2(\R^d)$. Indeed
\begin{multline*}
E\Big[ \int_{\R^d} | f \circ \phi_{t,s}^{-1}(x) |^2 dx \Big] 
=E\Big[ \int_{\R^d} | f(y) |^2 |\nabla  \phi_{t,s}(y) | dy \Big] 
\\
\leq  |f|_0^2  \sup_{x \in \R^d} E \Big[ \exp \Big\{ \int_t^s \mathrm{div} \beta_j( \phi_{s,r}(x) ) d\tilde{\Z}_r^j \Big\} \Big],
\end{multline*}
where we have used the change of variables $y = \phi_{t,s}(x)$ and the ``rough version'' of the Liouville lemma (see \cite{DFS}, Lemma 25),
$$
| \nabla  \phi_{t,s}(y)| = \exp \Big\{ \int_t^s \mathrm{div} \beta_j( \phi_{s,r}(y) ) d\tilde{\Z}_r^j \Big\} .
$$
The latter expression has finite expectation.

For a set of full $P$-measure we get the convergence
\begin{align*}
\sum\nolimits_{n \leq N} &\int_{\R^d \times \R^d} 
f(x) f(y) e_n(\phi_{t,T}(x)) e_n(\phi_{t,T}(y))
\\
&\quad \quad \quad \quad 
\times\exp \Big\{ \int_t^T \mathrm{div} \beta_j(\phi_{t,s}(x)) + \mathrm{div} \beta_j(\phi_{t,s}(y)) d\tilde{\Z}_s^j \Big\} dx dy
\\
 & = \sum\nolimits_{n \leq N} \left| \int_{\R^d  } f(x)  e_n(\phi_{t,T}(x))    \exp \Big\{ \int_t^T \mathrm{div} \beta_j(\phi_{t,s}(x)) d\tilde{\Z}_s^j \Big\} dx \right|^2 \\
 & = \sum\nolimits_{n \leq N} \left| \int_{\R^d  } f(\phi_{t,T}^{-1}(z))  e_n(z) dz \right|^2 
 \\
&\quad \quad \quad \quad \quad 
\underset{N\to\infty}{\longrightarrow} | f \circ \phi_{t,T}^{-1} |_0^2 = \int_{\R^d} |f(x)|^2  \exp \Big\{ \int_t^T \mathrm{div} \beta_j(\phi_{t,s}(x)) d\tilde{\Z}_s^j \Big\} dx .
 \end{align*}
The above convergence is monotone, and so taking the expectation yields exactly the claimed result.
\end{proof}

\subsection{The proof of Theorem \ref{EnergyEqualityTheorem}}

We have the following weighted energy equality
\begin{equation} \label{EnergyEquality}
 (u_t^2, m_t) + \int_0^t 2( |\nabla u_r|^2, m_r) dr = (u_0^2,m_0) - 2 \int_0^t ( F(u_r) , \nabla( u_r m_r)) dr .
\end{equation}
Indeed, consider \eqref{TensorInnerProduct} for $M_T^N(x,y) = \sum_{n \leq N} e_n(x) e_n(y)$;
\begin{equation} \label{ApproximateEnergyEquality}
(u_T^{\otimes 2}, M_T^N) + 2 \int_0^T (\nabla u_t \otimes \nabla u_t, M_t^N) dt = (u_0^{\otimes 2}, M_0^N) + 2 \int_0^T (u_t \st \mathrm{div}F(u_t), M_t^N) dt.
\end{equation}
Letting $N \rightarrow \infty$ we see that 
$$
(u_t^{\otimes 2}, M_t^N) \rightarrow (u_t^2, m_t).
$$
For any $t$ such that $\nabla u_t \in L^2(\R^d)$ we also get that 
\begin{align*}
(\nabla u_t \otimes \nabla u_t, M_t^N) \rightarrow (|\nabla u_t|^2, m_t)
\intertext{ and }
2 ( u_t \st \mathrm{div} F(u_t) , M_t^N) \rightarrow - 2 (F(u_t),\nabla(u_t m_t)),
\end{align*}
and we may now use dominated convergence to conclude that \eqref{MeasureChangedEnergyInequality} holds.

Now, from Proposition \ref{MeasureChange} and \eqref{EnergyEquality} we get 
\begin{equation} \label{MeasureChangedEnergyInequality}
|u_t|_0^2 + \int_0^t  |\nabla u_r|_0^2 dr  \lesssim |u_0|_0^2 +   \int_0^t ( F(u_r) , \nabla( u_r m_r)) dr  .
\end{equation}

\bigskip

We can now proceed to the proof of \eqref{FBoundedEnergy}.

We use the bound
\begin{multline*}
\int_0^t ( F(u_r) , \nabla( u_r m_r)) dr =  - \int_0^t ( \nabla F(u_r) u_r, u_r m_r) dr
\\
\leq  |\nabla F |_{\infty} \|m\|_{L^{\infty}([0,T] \times \R^d)} \int_0^t |u_r|_0^2dr  . 
\end{multline*}
From Gronwall Lemma, 
we have then
$$
|u_t|_0^2 + \int_0^t  |\nabla u_r|_0^2 dr  \lesssim |u_0|_0^2 \exp \left\{  t |\nabla F |_{\infty} \|m\|_{L^{\infty}([0,T] \times R^d)} \right\},
$$
and the result follows.

\bigskip

Finally, we prove \eqref{BurgersBoundedEnergy} for the classical Burgers non-linearity.

We rewrite
\begin{align*}
\int_0^t \frac{1}{2} ( u_r^2 , \nabla( u_r m_r)) dr = - \int_0^t \frac{1}{3} (u^3_r, \nabla m_r) dr \lesssim \| \nabla m \|_{\infty} \int_0^t \|u_r\|_{L^3}^3 dr .
\end{align*}

Recall the Gagliardo-Nirenberg inequalities 
$$
\| \phi \|_{L^3} \lesssim | \nabla \phi |_{0}^{1/6} | \phi |_{0}^{5/6} .
$$
Using Young's inequality $ab \leq \epsilon a^p + c_{\epsilon} b^{p'}$ with $p = 4$ and $p' = \frac{4}{3}$ gives

\begin{align*}
|u_t|_0^2 + \int_0^t  |\nabla u_r|_0^2 dr  &  \lesssim |u_0|_0^2 +  \| \nabla m\|_{\infty}  \int_0^t |\nabla u_r|_0^{1/2} |u_r|_0^{5/2} dr \\
&  \lesssim |u_0|_0^2 +  \| \nabla m\|_{\infty}  \epsilon \int_0^t |\nabla u_r|_0^{2}dr +  \| \nabla m\|_{\infty}  c_{\epsilon} \int_0^t |u_r|_0^{2q} dr,
\end{align*}
where $q := \frac{5}{3}$. If $\epsilon$ is small enough, we get
\begin{align*}
|u_t|_0^2 + \int_0^t  |\nabla u_r|_0^2 dr  &  \lesssim |u_0|_0^2 +  \| \nabla m\|_{\infty}  c_{\epsilon} \int_0^t |u_r|_0^{2q} dr .
\end{align*}

From the Bihari-LaSalle inequality (Theorem \ref{BihariTheorem}) we find that for all $t \leq ( \| \nabla m\|_{\infty}  c_{\epsilon} |u_0|_0^2)^{-1}$ we get
$$
|u_t|_0^2 \leq \frac{  |u_0|_0^2}{\left( 1 - (q-1) \| \nabla m\|_{\infty}  c_\epsilon  |u_0|_0^2 t \right)^{1/(q-1)}  } ,
$$
proving the theorem.
\hfill$\qed$

\section{Proof of Theorem \ref{MainTheorem}}
\subsection{Existence} \label{section:ProofMainTheoremExistence}

In this section we prove existence of a solution by taking smooth approximations of the rough path. It is well known that weak limits are not preserved by non-linear operations which hints that we need a strong compactness criterion. Moreover, there is the additional difficulty that we need to take a pointwise limit in time to show convergence of the rough integral. The latter convergence, however, can be taken weak in space since the rough integral is a linear operation of the solution. 

Now, we introduce a localization argument,
which is needed because of the lack of compactness in the embedding $H^1\hookrightarrow L^2(\R^d)$ on the whole space.
If $u$ is a solution to \eqref{MainEqURD}, for any open and bounded $U \subset \R^d$ define $\bar{u} = \tau_U^* u$ where $\tau_U : H^1_0(U) \rightarrow H^1(\R^d)$ is the extension operator given by $\tau_U f = f$ on $U$ and $\tau_U f = 0$ on $U^c$. 
Since $\tau_U$ is continuous, it is clear that $\bar{u}$ is $\alpha^*$-continuous with values in $(H^{1}_0(U))^*$ where $\alpha^*$ is as in Lemma \ref{aprioriLemma}. Moreover, since $u_t$ is actually in $L^2(\R^d)$ we have that $\bar{u}_t(x) = u_t(x)$ for almost all $x \in U$. Furthermore, for any smooth $g$ with compact support in $U$ we have
$$
(\bar{u}, \partial_j g) = (u, \partial_j g) = -( \partial_j u, g)
$$
so that $\nabla \bar{u}_t(x) = \nabla u_t(x)$ for almost all $x \in U$ and $ t \in [0,T]$. This gives that 
\begin{align*}
\sup_{ t \in [0,T]} |\bar{u}_t|_{L^2(U)}^2 + \int_0^T |\nabla \bar{u}_r |_{L^2(U)}^2 dr &  =\sup_{t \in [0,T]} |u_t|_{L^2(U)}^2 + \int_0^T |\nabla u_r |^2_{L^2(U)} dr \\
 & \leq \sup_{ t \in [0,T]} |u_t|^2_{L^2(\R^d)} + \int_0^T |\nabla u_r |^2_{L^2(\R^d)} dr ,
\quad \quad 
\end{align*}
which shows that the restriction $\bar{u}$ is in $L^{\infty}([0,T]; L^2(U)) \cap L^{2}([0,T]; H^1(U)) \cap C^{\alpha^*}([0,T]; (H^{1}_0(U))^*)$.

\bigskip

We now proceed to construct the solution via approximations of the rough path. 

Since $\Z$ is assumed to be geometric we know that there exists a sequence of smooth paths $Z(n)$ such that $\Z(n) \rightarrow \Z$ in the rough path metric. Denote by $u^n$ the solution of the corresponding equation, i.e.\
\begin{equation} \label{EqSmoothPath}
\partial_t u^n_t = \Delta u_t^n + \mathrm{div}F(u_t^n) + \beta_j \nabla u_t^n \dot{Z}^j_t(n) .
\end{equation}
By a solution to this equation we mean a mapping $u^n :[0,T] \rightarrow L^2(\R^d)$ such that $u^n :[0,T] \rightarrow H^{-2}$ is differentiable and \eqref{EqSmoothPath} is satisfied in $H^{-2}$.

We define the unbounded rough drivers
$$
A^1_{st}(n)\phi : = \beta_j \nabla \phi \delta Z_{st}^j(n) , \hspace{.5cm} A^2_{st}(n)\phi : = \beta_i \nabla  (\beta_j \nabla \phi ) \ZZ_{st}^{j,i}(n) ,
$$
and notice that we have
$$
\| A^i(n) - A^i \|_{ C_2^{i \alpha}([0,T];\mathcal{L}(H^k ; H^{k-i})) } \rightarrow 0 
$$
as $n \rightarrow \infty$.

We now show that a classical solution of the above equation also satisfies \eqref{MainEqURD}.

\begin{lemma}\label{Existence}
If $u^n : [0,T] \rightarrow L^2(\R^d)$ such that \eqref{EqSmoothPath} holds, then  
\begin{equation} \label{approxURD}
\delta u_{st}^n = \int_s^t \Delta u_r^n + \mathrm{div}F(u_r^n) dr + (A_{st}^1(n)  + A_{st}^2(n) )u_s^n + u_{st}^{n, \natural} ,
\end{equation}
and there exists constants $C$,$\alpha^* \in (0,1)$, $\zeta > 1$ and a final time $T_0$ independent of $n$ such that 
\begin{equation} \label{UniformRemainder}
| u_{st}^{n, \natural} |_{-3} \leq C|t-s|^{\zeta} ,
\end{equation}
and 
\begin{equation} \label{UniformHolder}
| \delta u_{st}^{n} |_{-1} \leq C|t-s|^{\alpha^*} ,
\end{equation}
for all $s,t \in [0,T_0]$.

\end{lemma}

\begin{proof}
Recall the discussion in section \ref{RoughPaths} and the two expressions for $u^{n, \natural}_{st}$ given by \eqref{approxURD} and \eqref{ExplicitRemainder}. We do the decomposition $\phi = J^{\eta} \phi + (I - J^{\eta}) \phi$ for any $\phi \in H^3$. From \eqref{approxURD} we see
\begin{align}
|(u^{n, \natural}_{st}, (I - J^{\eta}) \phi) | & \lesssim  |t-s| |(I - J^{\eta}) \phi|_2 +  |t-s| |(I - J^{\eta}) \phi|_1 + |t-s|^2 |(I - J^{\eta}) \phi|_2  \notag \\
 & \lesssim (  |t-s| \eta  +  |t-s|^2 \eta )|\phi|_3 \label{SmoothDriverRemainderEstimate1}
\end{align}
and from \eqref{ExplicitRemainder} we find 
\begin{align}
|(u^{n, \natural}_{st}, J^{\eta} \phi) | & \lesssim  |t-s|^3 |J^{\eta} \phi|_3 + |t-s|^3 |J^{\eta} \phi|_4 \notag \\
 & \lesssim |t-s|^3 \eta^{-1} | \phi|_3 \label{SmoothDriverRemainderEstimate2} .
\end{align}
Letting $\eta = |t-s|$ it follows that $u^{n ,\natural}$ is a remainder in $H^{-3}$.

The reader should notice that the inequalities \eqref{SmoothDriverRemainderEstimate1} and \eqref{SmoothDriverRemainderEstimate2} are \emph{not} uniform in $n$, and in fact they use the smoothness of $Z(n)$.

From Theorem \ref{EnergyEqualityTheorem} we get that there exists $T_0,C> 0$, independent of $n$ such that 
\begin{equation} \label{UniformEnergyEstimate}
\sup_{t \in [0,T_0]} |u_t^n |_0^2 + \int_0^{T_0} |\nabla u_t^n |_0^2 dt \leq C.
\end{equation}

Defining $\mu_t^n := \int_0^t \Delta u_r^n + \mathrm{div} F(u_r^n) dr$ we get 
\begin{equation} \label{UniformDriftEstimate}
[\mu^n]_{\mathrm{Lip}} \leq C,
\end{equation}
so that from Theorem \ref{aprioriEstimate} we get \eqref{UniformRemainder} and from Lemma \ref{aprioriLemma} we get \eqref{UniformHolder}. 
\end{proof}

Choose a sequence of open and bounded sets $U_m \subset U_{m+1}$ such that $\bigcup_m U_m  = \R^d$. For any $m$, the restriction of $u^n$ to $U_m$ is bounded in $C^{\alpha^*}([0,T_0]; H^{-1}(U_m)) \cap L^2([0,T_0]; H^1(U_m))$. Since $H^1(U_m)$ is compactly embedded into $L^2(U_m)$, it follows that $u^n$ is compact in $L^2([0,T_0]; L^2(U_m))$ and $C([0,T_0]; L^{2}(U_m)_w)$ from Lemma \ref{CompactnessLemma} and Lemma \ref{lem:weakc}.

We now find a subsequence which converges on compacts in the strong topology.

\begin{theorem} \label{ExistenceTheorem}
There exists a subsequence, $\{ u^{n_k} \}_{k \geq 1}$ of $\{ u^{n} \}_{n \geq 1}$ and an element $u \in L^2([0,T_0]; H^1(\R^d)) \cap C^{\alpha^*}([0,T_0]; H^{-1}(\R^d))$, such that for any compact $K \subset \R^d$ we have
$$
u^{n_k} \rightarrow u \hspace{.5cm} in \hspace{.5cm} L^2([0,T_0]; L^2(K)) \cap C([0,T_0]; L^{2}(K)_w) . 
$$
Moreover, the limit solves \eqref{MainEq}, in the sense of Definition \ref{def:solution}.
\end{theorem}

\begin{proof}[Proof of Theorem \ref{ExistenceTheorem} when $F$ is Lipshitz]
From \eqref{UniformEnergyEstimate} we may choose a subsequence $\{ u^{n_k} \}_{k \geq 1}$ of $\{ u^{n} \}_{n \geq 1}$ converging to some $u$ in the weak topology induced by $L^{\infty}([0,T_0]; L^2(\R^d)) \cap L^2([0,T_0]; H^1)$. From Lemma \ref{CompactnessLemma}, for every $m$ there exists a further subsequence $\{ u^{n_{k,m,j}} \}_{j \geq 1}$ of $\{ u^{n_k} \}_{k \geq 1}$ converging to some element $u^{(m)}$ in $L^2([0,T_0]; L^2(U_m)) \cap C([0,T_0]; L^{2}(U_m)_w)$. 

By uniqueness of weak and strong limits it is clear that $u$ and $u^{(m)}$ must coincide on $U_m$, which also gives that the full sequence $\{ u^{n_k} \}_{k \geq 1}$ is converging in $L^2([0,T_0]; L^2(U_m)) \cap C([0,T_0]; L^{2}(U_m)_w)$. The convergence statement of the theorem on compacts is now evident.

We are now ready to show that all the terms in the expansion \eqref{approxURD} converge. Fix $K$ compact in $\R^d$ and $\phi \in C^{\infty}(K)$. We get
\begin{itemize}

\item
$(u^{n_k}, \phi) \rightarrow (u, \phi)$ in $C([0,T_0])$

\item
$ (\Delta u^{n_k}, \phi)  \rightarrow  (\Delta u, \phi) $ in $C([0,T_0])$

\item
$(u^{n_k}, A_{}^{i,*}(n_k)\phi) \rightarrow (u, A_{}^{i,*}\phi)$ in $C_2([0,T_0])$

\item
$ \left|\int_0^t (F( u_r^{n_k}), \nabla \phi) dr  - \int_0^t (F( u_r), \nabla \phi) dr \right|  \leq \int_0^t |F( u_r^{n_k}) - F( u_r)|_{L^2(K)}  |\nabla \phi|_{H^1(K)} dr $ $ \rightarrow 0 . 
$
\end{itemize}

Since all the other terms in \eqref{approxURD} converge, also $(u_{st}^{n_k, \natural},\phi)$ must converge to some limit $(u_{st}^{\natural}, \phi)$. From \eqref{UniformEnergyEstimate}, \eqref{UniformDriftEstimate} and Theorem \ref{aprioriEstimate} we see that 
$$
|(u_{st}^{n_k, \natural},\phi)| \leq C |t-s|^{\zeta^*} |\phi|_{3}
$$
where $C$ is independent of $k$, so that the same estimate holds for $u_{}^{\natural}$.	Since $K$ was arbitrary, the result follows.
\end{proof}

We now proceed to the proof concerning the classical Burgers nonlinearity.
\begin{proof}[Proof of Theorem \ref{ExistenceTheorem} when $d=1$ and $F(u)=- \frac{1}{2}u^2$]
The only difference with the above proof is the convergence of the non-linearity.
But for any $r \in [0,T_0]$ such that $u_r^n, u_r \in H^1,$ we have
\begin{align*}
\big| ( u_r^n &\partial_xu_r^n - u_r \partial_x u_r, \phi) \big|
\\
& \leq \left| ( (u_r^n  - u_r) \partial_xu_r^n, \phi) \right| + \left| ( u_r^n -  u_r, \partial_x u_r  \phi) \right|
+ \left| ( u_r^n -  u_r,  u_r \partial_x \phi) \right| 
\\
& \leq  |u_r^n  - u_r|_0 | \partial_xu_r^n| |\phi|_{0, \infty}  + |  u_r^n -  u_r |_0 |\partial_x u_r|_0  |\phi|_{0, \infty}
+ |  u_r^n -  u_r |_0 |u_r|_0 |\partial_x \phi |_{0, \infty} 
\\
& \lesssim \big( |u^n_r  - u_r|_0 | \partial_xu_r^n|  + |  u_r^n -  u_r |_0 |\partial_x u_r|_0
+ |  u_r^n -  u_r |_0 |u_r|_0 \big) |\phi|_{1, \infty} .
\end{align*}
Integrating over $r$ and sending $n$ to infinity gives the result. 
\end{proof}

\subsection{Uniqueness} \label{section:ProofMainTheoremUniqueness}

We now prove uniqueness in the class $C([0,T]; L^2(\R^d)) \cap L^2([0,T]; H^1)$. The steps are very similar to the one's in the previous section. Assume $u^{(1)}$ and $u^{(2)}$ are two solutions to \eqref{MainEq} and define $v := u^{(1)} - u^{(2)}$. We then get
$$
\delta v_{st} = \int_s^t \big[\Delta v_r + \mathrm{div}( F(u_r^{(1)}) - F(u_r^{(2)})) \big]dr + A_{st}^1 v_s + A_{st}^2 v_s + v_{st}^{\natural}
$$
where we have defined  
$$
v_{st}^{\natural} := u_{st}^{1,\natural} - u_{st}^{2,\natural} . 
$$

Since the linear terms in the above equation have the same structure as in the previous section, we may use the same weight function, $m$, as in Proposition \ref{MeasureChange}.

\begin{lemma} \label{SquareEquation}
Suppose $\Z$ is a geometric rough path and $\beta_j \in C_b^{6}(\R^d)$ for all $j$. Then we have the following weighted energy equality
\begin{equation} \label{EnergyForDifference}
 (v_t^2, m_t) + \int_0^t 2( |\nabla v_r|^2, m_r) dr = (v_0^2,m_0) - 2 \int_0^t ( F(u^1_r) - F(u^2_r) , \nabla( v_r m_r)) dr .
\end{equation}
\end{lemma}

We consider again two different cases.

\begin{theorem}\label{ContinuityWithFTheorem}
Assume $F$ has a bounded derivative. Then we have the following estimate
\begin{equation} \label{ContinuityWithFEquation}
\sup_{ 0 \leq t \leq T}|u_t^{(1)} - u_t^{(2)}|^2_0 + \int_0^T |\nabla u_r^{(1)} - \nabla u_r^{(2)} |_0^2 dr \lesssim |u_0^{(1)} - u_0^{(2)}|^2_0 .
\end{equation}
In particular,  uniqueness holds in the space $C([0,T]; H^0) \cap L^2([0,T]; H^1)$ for any $T>0,$ which amounts to say that finite-energy solutions are unique.
\end{theorem}

\begin{proof}
From Proposition \ref{MeasureChange} we have that 
\begin{align}
\label{bounds:u_m}
|\nabla(vm)|_0 \lesssim |v|_0 + |\nabla v|_0
\intertext{and} 
\label{bounds:u_m_2}
\int_0^t |\nabla v_r|_0^2 dr \lesssim \int_0^t (|\nabla v_r|_0^2, m_r) dr\,,
\end{align}
which gives 
\begin{align*}
|v_t|^2_0  + \int_0^t  |\nabla v_r|^2_0 dr & \lesssim (v_t^2, m_t) + \int_0^t 2( |\nabla v_r|^2, m_r) dr
\\
 & = (v_0^2,m_0) - 2 \int_0^t ( F(u^{(1)}_r) - F(u^{(2)}_r) , \nabla( v_r m_r)) dr
\\
 & \lesssim |v_0|^2_0  +  \int_0^t | F(u^{(1)}_r) - F(u^{(2)}_r)|_0 | \nabla( v_r m_r) |_0 dr
\\
 & \lesssim |v_0|^2_0  +  |\nabla F|_{\infty} \int_0^t | v_r|_0 (|v_r|_0 + |\nabla v_r|_0) dr
\\
 & \lesssim |v_0|^2_0  +  \epsilon |\nabla F|_{\infty} \int_0^t | \nabla v_r|_0^2 dr + C_{\epsilon}  |\nabla F|_{\infty} \int_0^t |v_r|_0^2   dr
\end{align*}
where we have used Young Inequality in the last step. Now choosing $\epsilon$ small enough, we get
\begin{align*}
|v_t|^2_0  + \int_0^t  |\nabla v_r|^2_0 dr & \lesssim |v_0^2|_0  + C_{\epsilon}  |\nabla F|_{\infty} \int_0^t |v_r|_0^2dr    .
\end{align*}
The result follows from Gronwall Lemma.
\end{proof}

Similarly, we get uniqueness of the finite-energy solutions of the classical Burgers equation.

\begin{theorem} \label{ContinuityBurgersTheorem}
The classical Burgers equation admits the following local Lipschitz estimate
\begin{multline*}
\sup_{0 \leq t \leq T}|u_t^{(1)} - u_t^{(2)}|^2_0  + \int_0^T  |\partial _x(u_r^{(1)} - u_r^{(2)})|^2_0 dr 
\\
\lesssim |u_0^{(1)} - u_0^{(2)}|^2_0 \exp\Big\{ C \int_0^T  ( |\partial _xu_r^{(1)} |_0 + |u^{(2)}_r|_0 + |\partial _x u_r^{(2)}|_0)^{4/3}, dr \Big\}
\end{multline*}
for every $T$ such that $u^{(i)}$ is a solution on $[0,T]$ for $i=1,2,$ in the sense of Definition \ref{def:solution}.

In particular, finite-energy solutions are unique.
\end{theorem}

\begin{proof}
With $F(u) = - \frac{1}{2} u^2$ we rewrite \eqref{EnergyForDifference} as:
\begin{equation} 
 (v_t^2, m_t) + \int_0^t 2( |\partial _x v_r|^2, m_r) dr = (v_0^2,m_0)  - 2  \int_0^t ( u_r^{(1)} \partial _xu^{(1)}_r - u_r^{(2)} \partial _x u^{(2)}_r ,v_r m_r) dr ,
\end{equation}
where as before we denote by 
$
v_r:=u^{(1)}-u^{(2)}.
$

From \eqref{bounds:u_m}, \eqref{bounds:u_m_2} and integration by parts, we have this time
\begin{align*}
|v_t|^2_0  & + \int_0^t  |\partial _x v_r|^2_0 dr\\
 & \lesssim (v_0^2,m_0) -  2  \int_0^t ( u_r^{(1)} \partial _x v_r + v_r \partial _x u^{(2)}_r ,v_r m_r) dr   \\
 & = (u_0^2,m_0) +    \int_0^t (  (v_r)^2, \partial _x(u_r^{(1)} m_r)) dt + 2 \int_0^t ( (v_r)^2  \partial _x u^{(2)}_r , m_r) dr .
\end{align*}
Next, interpolating $L^4$ between $L^2$ and $H^1$ yields
\begin{align*} 
|v_t^2|_0  & + \int_0^t  |\partial _x v_r|^2_0 dr\\ 
 & \lesssim |v_0|_0^2  +  \int_0^t |v_r|_{L^4}^2 \big( |\partial _x u_r^{(1)} |_0 + |u^{(2)}_r|_0 + |\partial _x u_r^{(2)}|_0\big) dr  \\
& \lesssim |v_0|_0^2  +  \int_0^t |\partial _x v_r|_0^{1/2} |v_r|_0^{3/2} \big( |\partial _x u_r^{(1)} |_0 + |u^{(2)}_r|_0 + |\partial _x u_r^{(2)}|_0\big) dr  \\
&  \lesssim |v_0|_0^2  +  \epsilon \int_0^t |\partial _x v_r|_0^{2}dr + C_{\epsilon} \int_0^t |v_r|_0^{2} \big( |\partial _x u_r^{(1)} |_0 + |u^{(2)}_r|_0 + |\partial _x u_r^{(2)}|_0\big)^{4/3} dr .
\end{align*}
Choosing $\epsilon$ small enough we get
$$
|v_t^2|_0  + \int_0^t  |\partial _x v_r|^2_0 dr  \lesssim |v_0|_0^2  +  C_{\epsilon} \int_0^t |v_r|_0^{2} ( |\partial _x u_r^{(1)} |_0 + |u^{(2)}_r|_0 + |\partial _x u_r^{(2)}|_0)^{4/3} dr .
$$
Since $r \mapsto ( |\partial _x u_r^{(1)} |_0 + |u^{(2)}_r|_0 + |\partial _x u_r^{(2)}|_0)^{4/3}$ is integrable over $[0,T]$ by assumption, we get from Gronwall Lemma
$$
|v_t|^2_0  + \int_0^t  |\partial _x v_r|^2_0 dr
\lesssim |v_0|^2_0 \exp\Big\{ C \int_0^t  ( |\partial _x u_r^{(1)} |_0 + |u^{(2)}_r|_0 + |\partial _x u_r^{(2)}|_0)^{4/3} dr \Big\},
$$
which shows the claimed estimate. 
\end{proof}

\appendix

\section{Appendix} 

\subsection{Compact embedding results}

The following compact embedding result is comparable to the fractional version of the  Aubin-Lions compactness result (\cite{Aubin}, \cite{Lions}).

\begin{lemma} \label{CompactnessLemma}
Suppose $\bH_1 \subset \bH \subset \bH_{-1}$ are Banach spaces such that $\bH_1$ is compactly embedded into $\bH$ and $\bH_1$ and $\bH_{-1}$ are reflexive.

For any $\kappa > 0$ the set 
	$$
	X = L^2([0,T];\bH_1) \cap C^{\kappa}([0,T]; \bH_{-1} )$$
	is compactly embedded into 
	$L^2([0,T]; \bH)$.
	\end{lemma}

\begin{proof}

Notice that $C^{\kappa}([0,T]; \bH_{-1}) \subset W^{ \kappa', 2}([0,T]; \bH_{-1})$ for any $\kappa < \kappa'$, where $W^{ \kappa', 2}([0,T]; \bH_{-1})$ denotes the set of functions $g$ such that
$$
\int_{[0,T]^2} \frac{ |\delta g_{st}|^2}{|t-s|^{1 + 2 \kappa' }} dt ds < \infty .
$$
From \cite[Theorem 2.1]{flandoli1995martingale} we get that $X$ is compactly embedded in $L^2([0,T]; \bH)$.
\end{proof}

To take the limit in the rough terms of \eqref{MainEq} we need to have a limit pointwise in time. Since the noise is linear, it is enough to have weak limits in space.

\begin{lemma}\label{lem:weakc}
With the same assumptions as above, and suppose $\bH$ is a Hilbert space and there exists a continuous bilinear map $(\cdot, \cdot) : \bH_{-1} \times \bH_1 \rightarrow \R$ that coincides with the inner-product on $\bH$ when restricted to $ \bH \times \bH_1$. Suppose in addition that $\bH_1$ is separable and dense in $\bH$. Then the set 
	$$
	Y=L^\infty([0,T];\bH )\cap C^{\kappa}([0,T]; \bH_{-1} )
	$$
is compactly embedded into  $C([0,T];\bH_w)$, the space of weakly continuous functions with values in $\bH$. 
\end{lemma}

\begin{proof}
Let $g\in Y$. First, we will show that for all $\varphi\in \bH$ the mapping
\begin{equation}\label{eq:cont}
t\mapsto(g_t,\varphi)\in C([0,T]).
\end{equation}
To this end, we observe that since $g\in L^\infty([0,T];\bH )$ it follows that there exists $R>0$ such that $g_t\in B_R$ for all $t\in[0,T]$, where $B_R\subset \bH$ is a ball of radius $R$.
Take any family $(h_n)_{n\in \mathbb{N}}$ that belong to $\bH_{1}$ and their finite linear combinations are dense in $\bH$. Then
\begin{align}\label{eqcon}
\left| (g_t, \varphi)-  (g_s,\varphi)\right| &\leq 
\left| \left(g_t-g_s, \sum\nolimits_{n\leq M} \beta_n h_n  \right)\right| + 
\left| \left(g_t-g_s,\varphi - \sum\nolimits_{n\leq M} \beta_n h_n ,  \right)\right| \nonumber\\ 
&\leq 
\left| \left(g_t-g_s, \sum\nolimits_{n \leq M} \beta_n h_n , \right)\right| 
+ R \left| \varphi - \sum\nolimits_{n \leq M} \beta_n h_n \right|_{0}\nonumber\\
&\leq c(M) |t-s|^\kappa+ R \left|\varphi - \sum\nolimits_{n \leq M} \beta_n h_n \right|_{0}, 
\end{align}
where the last term can be made small uniformly for all $s,t \in [0,T]$ by taking suitable $\beta_{{m}}$ and $M$ large enough. Hence \eqref{eq:cont} follows. The compactness of the embedding follows from the  abstract Arzel\` a--Ascoli theorem. Indeed the ball $B_R$ is relatively weakly 
compact, and the desired equi-continuity follows from (\ref{eqcon}).
\end{proof}

\subsection{Bihari-LaSalle}

We recall the Bihari-LaSalle inequality, which is a non-linear generalization of the Gronwall's lemma, due to \cite{Bihari} and \cite{LaSalle}.

\begin{theorem} \label{BihariTheorem}
Suppose $q > 1$ and two positive functions $x,k : [0,T] \rightarrow \R_+$ satisfy
$$
x_t \leq x_0 + \int_0^t k_s x_s^q ds .
$$
Then for all $t$ such that $(q-1) x_0 \int_0^t k_s ds < 1$ we have the following estimate
$$
x_t \leq \frac{x_0}{ \left( 1 -(q-1)x_0\int_0^t k_s ds \right)^{1/(q-1)} } .
$$

\end{theorem}

\begin{proof}
By hypothesis
$$
\frac{ x_t}{ x_0 + \int_0^t k_s x_s^q ds } \leq 1
$$
which gives
$$
\frac{ k_t x_t^q}{\left( x_0 + \int_0^t k_s x_s^q ds \right)^q } \leq k_t . 
$$
The above right hand side is the time derivative of $(1- q)^{-1} \left(x_0 + \int_0^t k_s x_s^q ds \right)^{-q+1}$ so that 
$$
(1- q)^{-1} \left(x_0 + \int_0^t k_s x_s^q ds \right)^{-q+1} \leq (1- q)^{-1} x_0^{-q + 1} + \int_0^t k_s ds .
$$
Since $q > 1$ we get
$$
x_0 + \int_0^t k_s x_s^q ds \leq  \left( x_0^{-q + 1} + (1-q)\int_0^t k_s ds \right)^{1/(1-q)} .
$$
The result follows.
\end{proof}

\subsection{Sewing Lemma}

For the reader's convenience, we recall the Sewing Lemma. Other formulations of this classical result can be found, e.g.\ in \cite{FH14} or \cite{Gubinelli04}.

\begin{lemma}[\textbf{Sewing lemma}] \label{sewingLemma}
Assume $G : \Delta(T) \rightarrow E$ be such that 
$$
|\delta G_{s \theta t}| \lesssim |t-s|^{a}, \quad 0 \leq s \leq \theta \leq t \leq T
$$ 
for some $a > 1$, and denote by $[\delta G]_a$ the smallest constant satisfying the above bound. 

Then there exists a unique pair $I: [0,T] \rightarrow E$ and $I^{\natural} : \Delta(T) \rightarrow E$ satisfiying
$$
\delta I_{st} = G_{st} + I_{st}^{\natural}
$$
where 
$$
|I_{st}^{\natural}| \leq C_a [\delta G]_a |t-s|^{a}.
$$
for some constant $C_a$ only depending on $a$. In fact, $I$ is defined via the Riemann type integral approximation
\begin{equation} \label{RiemannSum}
I_t = \lim_{ |\pi| \rightarrow 0} \sum_{ (u,v) \in \pi } G_{u v} .
\end{equation}
The above limit is taken along any sequence of partitions $\pi$ of $[0,t]$ whose mesh converges to 0.

Moreover, if $|G_{st}| \lesssim |t-s|^{b}$ for some $b>0$ we have $|\delta I_{st}| \lesssim |t-s|^{b}$.

\end{lemma}

As a corollary we get the following. 
\begin{corollary} \label{SewingCorollary}
Retain the assumptions in the previous lemma. Then, if $|G_{st}| \lesssim |t-s|^b$ for some $b > 1$ we have
$$
[G]_a \leq C_a [\delta G]_a
$$
for some constant $C_a$ only depending on $a$.
\end{corollary}

\begin{proof}
Construct $(I, I^{\natural})$ as in Lemma \ref{sewingLemma}. Since $b>1$ we necessarily have $I=0$ and the result follows. 
\end{proof}

\end{document}